\documentclass[onefignum,onetabnum]{siamart171218}

\usepackage{lipsum}
\usepackage{amsfonts}
\usepackage{graphicx}
\usepackage{epstopdf}
\usepackage{algorithmic}
\ifpdf
  \DeclareGraphicsExtensions{.eps,.pdf,.png,.jpg}
\else
  \DeclareGraphicsExtensions{.eps}
\fi

\newsiamremark{remark}{Remark}
\newsiamremark{hypothesis}{Hypothesis}
\crefname{hypothesis}{Hypothesis}{Hypotheses}
\newsiamthm{claim}{Claim}

\crefname{equation}{}{}
\crefname{subsection}{Section}{Sections}

\headers{Upscaling errors for the Landau-Lifshitz equation}{L. Leitenmaier and O. Runborg}

\title{Upscaling errors in Heterogeneous Multiscale Methods for the Landau-Lifshitz equation 
  }

\author{Lena Leitenmaier\thanks{Department of Mathematics, KTH,
    Royal Institute of Technology, Stockholm, Sweden
    (\email{lenalei@kth.se}).}  \and Olof Runborg\thanks{Department
    of Mathematics, KTH, Royal Institute of Technology, Stockholm,
    Sweden (\email{olofr@kth.se}).}  }

\usepackage{amsopn}

\ifpdf
\hypersetup{
  pdftitle={Upscaling errors in Heterogeneous Multiscale Models for the Landau-Lifshitz equation},
  pdfauthor={L. Leitenmaier, O. Runborg}
}
\fi

\usepackage{xspace}
\usepackage{subcaption}

\newcommand{\M}{\ensuremath{\mathbf{M}}\xspace}
\newcommand{\m}{\ensuremath{\mathbf{m}}\xspace}

\newcommand{\K}{\ensuremath{\mathcal{K}}\xspace}

\newcommand{\g}{\ensuremath{\mathbf{g}}\xspace}
\newcommand{\Real}{\ensuremath{\mathbb{R}}\xspace}

\newcommand{\w}{\ensuremath{\mathbf{w}}\xspace}
\renewcommand{\u}{\ensuremath{\mathbf{u}}\xspace}

\newcommand{\f}{\ensuremath{\mathbf{f}}\xspace}
\newcommand{\h}{\ensuremath{\mathbf{h}}\xspace}

\renewcommand{\b}{\ensuremath{\mathbf{b}}\xspace}

\newcommand{\e}{\ensuremath{\mathbf{e}}\xspace}
\newcommand{\F}{\ensuremath{\mathbf{F}}\xspace}

\renewcommand{\v}{\ensuremath{\mathbf{v}}\xspace}

\newcommand{\A}{\ensuremath{\mathbf{A}}\xspace}
\newcommand{\I}{\ensuremath{\mathbf{I}}\xspace}

\renewcommand{\L}{\ensuremath{\mathcal{L}}}
\renewcommand{\AA}{\ensuremath{\mathcal{A}}}
\newcommand{\bnabla}{\boldsymbol\nabla}

\renewcommand{\Re}{\ensuremath{\mathcal{R}e}\xspace}
\renewcommand{\Im}{\ensuremath{\mathcal{I}m}\xspace}

\begin{document}

\maketitle

\begin{abstract}
  In this paper, we consider several possible ways to set up
  Heterogeneous Multiscale Methods for the Landau-Lifshitz equation
  with a highly oscillatory diffusion coefficient, which can be seen
  as a means to modeling rapidly varying ferromagnetic materials. We then prove
  estimates for the errors introduced when approximating the
  relevant quantity in each of the models given a periodic problem,
  using averaging in time and space of the solution to a
  corresponding micro problem. In our setup, the Landau-Lifshitz
  equation with highly oscillatory coefficient is chosen as the
  micro problem for all models. We then show that the averaging
  errors only depend on $\varepsilon$, the size of the microscopic
  oscillations, as well as the size of the averaging domain in time
  and space and the choice of averaging kernels.

\end{abstract}



\section{Introduction}

In micromagnetics, the evolution of the magnetization within a
ferromagnet is described by the Landau-Lifshitz (LL) equation
\cite{LandauLifshitz35}.  In this paper, we consider a
simplification of the deterministic version of this equation, where
we only take into account the exchange interaction between magnetic
moments and neglect other contributions influencing the
magnetization, such as anisotropy, temperature and external
field. We consider a ferromagnet with a rapidly varying material,
which we model by introducing a material coefficient
$a^\varepsilon(x)$, where $\varepsilon \ll 1$ represents the spatial
scale of the finest variations.. One example could be a composite,
consisting of two different materials with different interaction
behavior and layers of thickness $\varepsilon$.
According to this simplified model, the partial
differential equation determining the evolution of the magnetization
$\M^\varepsilon(x, t)$ then is
\begin{equation}
  \label{eq:ll}
  \begin{split}
    \partial_t \M^\varepsilon(x,t) &= - \M^\varepsilon(x,t) \times \L \M^\varepsilon(x, t) - \alpha \M^\varepsilon(x,t) \times \M^\varepsilon(x,t) \times \L \M^\varepsilon(x, t) \,, \\
    \M^\varepsilon(x, 0) &= \M_\mathrm{init}(x),
\end{split}
\end{equation}
where
$\M_\mathrm{init}(x)$ is a smooth function with values in $\Real^3$ such that
$|\M_\textrm{init}(x)| = 1$, and $0 < \alpha \le
1$ a damping coefficient. In this model, the effective field is
given by
\begin{equation}
  \label{eq:H_eff}
  \L \M^\varepsilon := \bnabla \cdot \left(a^\varepsilon(x) \bnabla \M^\varepsilon(x, t)\right).
\end{equation}
Here the coefficient
$a^\varepsilon$ influences the overall behavior of the magnetization
significantly. A very similar model was first introduced in
\cite{hamdache} and used recently in
\cite{alouges2019stochastic}. Also in for example \cite{multilayer}
and \cite{highcontrast}, related approaches are applied.

When solving \cref{eq:ll} numerically, one would have to resolve the
$\varepsilon$-scale in order to get a correct result. However, the
resulting amount of computational work is infeasible for small
$\varepsilon$. Instead of solving \cref{eq:ll}, one therefore in
many cases considers solutions to a corresponding effective equation
instead, which capture the correct magnetization behavior on a
coarse scale but do not resolve the
$\varepsilon$-scale. For periodic problems, one can apply techniques
from classic homogenization theory, \cite{Cioranescu_Donato_1},
\cite{Bensoussan_Lions_Papanicolaou_1}, to obtain such a homogenized
solution
$\M_0$ as well as correction terms as shown in \cite{paper1}. When
aiming to deal with somewhat more general coefficients, though, it
can be advantageous to instead use numerical methods to approximate
the homogenized solution.  This can be done using multiscale methods
like equation free methods \cite{kevrekidis2003} or heterogeneous
multiscale methods (HMM) \cite{weinan2003}, \cite{hmm1},
\cite{acta_numerica}.  The basic idea of HMM is to combine a coarse
scale macro model, that involves some unknown quantity, with micro
problems that are solved on a short time interval and
small domain only. In the so-called upscaling process, the solution
from the micro problem is then averaged to obtain the quantity that
is needed to complete the macro model.  This is the approach that we
consider in this paper.  In particular, we choose three different
HMM macro models for \cref{eq:ll}. For the case of a periodic material
coefficient, we then investigate the upscaling error for each of the
models, in order to get an understanding of what are good ways to
set up HMM for this problem. We come to the conclusion that all
three models give very similar results and can thus be valid choices
for HMM setups. Which model to choose can thus mostly be based on
advantages related to the numerical implementation.

HMM has previously been applied to a Landau-Lifshitz
problem in \cite{arjmand1}, \cite{arjmand2}. However, in these
articles, the authors do not consider the case with a material coefficient that is highly
oscillatory in space, \cref{eq:ll}, but instead a highly
oscillatory external field with temporal oscillations.

In the remainder of this section, we shortly introduce some of the
notation that is used in the following. We furthermore describe the
homogenized solution for \cref{eq:ll} with a periodic material
coefficient as derived in \cite{paper1}, which will subsequently act
as a reference.  We continue in \Cref{sec:hmm} by describing the concept of
heterogeneous multiscale methods as well as the models considered.
In \Cref{sec:hom}, estimates for the homogenized solution and
the corresponding correctors to the HMM micro problem are stated
to provide the basis that is required for the
subsequent derivations. We moreover add an explicit description of a
particular correction term. In \Cref{sec:avg}, we derive several
lemmas regarding the averaging required for numerical
homogenization. These lay the ground for the error estimates for the
different upscaling-models, which are given in \Cref{sec:avg_ll} and
constitute the main result of this paper. Finally, in \Cref{sec:num} we present
several numerical examples in one and two space dimensions which
show the validity of the theoretical estimates.

\subsection{Preliminaries}
We consider $\ell$-periodic solutions to \cref{eq:ll} in the
$d$-dimen\-sional hypercube $\Omega = [0, \ell]^d$ and time interval
$[0, T]$, for some $\ell > 0$.  We use $Y$ to denote the
$d$-dimensional unit cell $[0, 1]^d$ and let
$\Omega_\mu := [-\mu, \mu]^d$ for a parameter $\mu$.

We denote by $H^q(\Omega)$ the standard periodic Sobolev spaces on
$\Omega$ and by $H^{q, p}(\Omega; Y)$ periodic Bochner-Sobolev
spaces on $\Omega \times Y$.  The corresponding norms are
$\|\cdot\|_{H^q}$ and $\|\cdot\|_{H^{q, p}}$. 
%
Furthermore, $W^{1, q}$ and $W^{p, \infty}$ denote standard Sobolev
spaces, $\|\cdot\|_{W^{p, \infty}}$ being the
Sobolev supremum norm on $\Omega$: given
$\partial_x^\beta u \in L^\infty(\Omega)$ for a multi-index $\beta$
with $0 \le |\beta| \le p$, it holds that
\[\|u\|_{W^{p, \infty}} = \max_{|\beta| \le p} \sup_{x \in \Omega} |\partial_x^\beta u|\,.\]
Furthermore, we make frequent use of the Sobolev inequality
stating that given $u\in H^2(\Omega)$ for dimension $d \le 3$, it holds that
\begin{align}
  \label{eq:sob}
\sup_{x\in \Omega} |u(x)| \le C \|u\|_{H^2(\Omega)}\,.
\end{align}

In general, we use capital letters to refer to solutions on the
whole domain $\Omega$ and for time $[0, T]$,
$\M: \Omega \times [0,T] \to \Real^3$. When instead considering a
micro problem set on $\Omega_\mu$, we use lower case letters to denote the solution.  By
$\bnabla \M$ we denote the Jacobian matrix of the vector-valued
function $\M$.
We assume in general that scalar and cross product between a
vector-valued and a matrix-valued function are done column-wise,
while the divergence operator is applied row-wise.

The differential operator $L$ is defined such
that for $u \in H^2(\Omega)$,
\begin{equation}
  \label{eq:L}
  L u(x) = \nabla \cdot (a^\varepsilon(x) \nabla u(x)),
\end{equation}
where $a^{\varepsilon}$ is a highly oscillatory, smooth, scalar
coefficient function.  Moreover, we denote by $\L$ the corresponding
operator acting on vector-valued functions in $\Real^3$,
\begin{align*}
  \L \m =
  \begin{bmatrix}
   L m^{(1)} &
   L m^{(2)} &
   L m^{(3)}
  \end{bmatrix}^T\,,
\end{align*}
where $m^{(i)}$ is the $i$-th component of $\m$.

\subsection{Homogenized equation for the periodic problem}
The homogenization of \cref{eq:ll} with a periodic material
coefficient was studied in \cite{paper1}. Using the setup considered
there makes it possible to obtain bounds for the errors introduced
when using HMM. It therefore is the scenario that we focus on in the
subsequent proofs.  Specifically, we assume that $\M^\varepsilon$
satisfies \cref{eq:ll} with $a^\varepsilon = a(x/\varepsilon)$ on a
domain $\Omega = [0, \ell]^d \subset \Real^d$, where $d = 1, 2$ or 3
and $\ell > 0$, and for $0 \le t \le T$.  Then it is shown in
\cite{paper1} that the corresponding homogenized equation
is
\begin{subequations} \label{eq:hom_macro}
\begin{align}
  \partial_t \M_0 &= - \M_0 \times \bnabla \cdot (\bnabla \M_0 \A^H) - \alpha
\M_0 \times \M_0 \times \bnabla \cdot (\bnabla \M_0 \A^H) \,, \\
  \M_0(x, 0) &= \M_\mathrm{init}(x)\,,
\end{align}
\end{subequations}
 for
$0 \le t \le T$ and $x \in \Omega$.
The constant homogenized coefficient matrix $\A^H \in \Real^{d \times d}$
is as in elliptic homogenization theory given by
\begin{align}\label{eq:AH}
  \A^H := \int_Y a(y) \left( \I + (\bnabla_y \boldsymbol \chi)^T \right) dy\,,
\end{align}
where $\boldsymbol \chi: \Real^d \to \Real^d$ is the solution to the elliptic {cell problem}
\begin{equation}
  \label{eq:cell_problem}
  \bnabla_y \cdot(a(y) \bnabla_y \boldsymbol \chi(y)) = - \nabla_y a(y) \,.
\end{equation}
Note that \cref{eq:cell_problem} determines $\boldsymbol \chi$ only up to a
constant. Throughout this article, we assume that this constant is chosen
such that $\boldsymbol \chi$ has zero average.

For the difference between $\M^\varepsilon$ and $\M_0$, the
following result was proved in \cite{paper1}.
\begin{theorem}\label{thm:macro_err}
  Assume that $\M^\varepsilon \in C^1([0, T]; H^2(\Omega))$ is a
  classical solution to \cref{eq:ll} with a periodic material
  coefficient $a^\varepsilon = a(x/\varepsilon)$ where
  $a \in C^\infty(\Omega)$ and that
  $a_\mathrm{min} \le a(x) \le a_\mathrm{max}$ for some constants
  $a_\mathrm{min} , a_\mathrm{max} > 0$. Assume that there is a
  constant $M$ independent of $\varepsilon$ such that
  $\|\bnabla \M^\varepsilon(\cdot, t)\|_{L^\infty} \le M$ for
  $0 \le t \le T$.  Moreover, suppose that
  $\M_0 \in C^\infty([0, T]; H^\infty(\Omega))$ is a classical
  solution to \cref{eq:hom_macro}.  We then
  have 
  \begin{align}\label{eq:main_l2_m0}
  \|\M^\varepsilon(\cdot, t) - \M_0(\cdot, t)\|_{L^2} \le C \varepsilon\,,
    \qquad 0 \le t \le T\,,
  \end{align}
where the constant $C$ is independent of $\varepsilon$ and $t$ but
depends on $M$ and $T$.
\end{theorem}

\section{Heterogeneous multiscale methods}\label{sec:hmm}
The concept of heterogeneous multiscale methods was first introduced
by E and Engquist in \cite{weinan2003}. It provides a general
approach to treat multiscale problems with scale separation, where a
description of the microscopic problem is available but would be too
computationally expensive to use throughout the whole domain.
The idea is therefore to use numerical homogenization with the goal
to get a good approximation to the effective solution of the
original problem.  In general, HMM models involve three parts:
\begin{enumerate}
\item \emph{Macro model}: an incomplete model for the whole computational domain,
  discretized with a coarse grid, that is set up in such a way that some data is missing.
\item \emph{Micro model}: an exact model discretized with a grid
  resolving the fine $\varepsilon$-scale, which is however only
  solved on a small domain, where it is feasible to use the 
  expensive description.
\item \emph{Upscaling}: an averaging procedure that uses the data
  obtained when solving the micro problem to generate the quantity
  needed to complete the macro model.
\end{enumerate}
It is important to make sure that micro and macro model
are consistent, which is typically achieved by choosing the initial
data for the micro problem as a restriction of the current macro
solution.

The HMM framework has been successfully applied to a wide range of applications;
see for instance the surveys in \cite{hmm1,acta_numerica}.
In this paper, we aim to find a good way to set up HMM for the
Landau Lifshitz equation \cref{eq:ll}.

In general, the error in the HMM solution consists of two major
components apart from discretization errors: an error term related
to the fact that the solution to the effective equation is
approximated instead of the original one and the so-called HMM
error. The HMM error in turn depends on the upscaling error, the
error introduced in the data estimation process
\cite{acta_numerica}. In case of the Landau-Lifshitz equation with a
periodic material coefficient, an estimate for the first error term
is given by \Cref{thm:macro_err}.  The $L^2$-homogenization error is
$\mathcal{O}(\varepsilon)$.  In this paper, we focus on estimates
for the upscaling error and investigate how it is influenced by
different choices of HMM-models.

We consider three different setups. All three 
are based on the same micro model, the full Landau-Lifshitz equation,
\begin{subequations}\label{eq:micro_prob}
  \begin{align}
    \partial_t \m^\varepsilon &= - \m^\varepsilon \times \L \m^\varepsilon - \alpha \m^\varepsilon \times \m^\varepsilon \times \L \m^\varepsilon
                                \,, \quad x \in \Omega, ~0 < t \le \eta, \\
    \m^\varepsilon(0, x) &= \m_\mathrm{init}(x) = \mathcal{R}(\M)\,,
\end{align}
\end{subequations}
with periodic boundary conditions. The initial data
$\m_\mathrm{init}$ is assumed to be a restriction $\mathcal{R}$ of
the macro data $\M$ such that the micro and macro model are
consistent and $|\m_\mathrm{init}| \equiv 1$.  One possible choice
to obtain such initial data is by using a normalized interpolation
polynomial based on the macro data.
In this paper, we assume a solution to the micro problem in the
whole domain $\Omega$ with periodic boundary conditions. In
practice, one would solve the micro problem only on a small domain
$[-\mu', \mu']^d$, where $0 < \mu' \ll \ell$.  However, this
requires a choice of boundary conditions which introduce some
additional error. To simplify the following analysis, we avoid
dealing with this issue here and assume a solution in the whole
domain.
For the averaging we then consider only the solution in a box
$\Omega_\mu = [-\mu, \mu]^d$ in space and an interval $[0, \eta]$ in
time, where $\mu \sim \varepsilon$ and $\eta \sim \varepsilon^2$
. This matches the scales of the fast variations in the problem as
explained in \cite{paper1}.

The other two HMM components, macro model and upscaling, differ
between the models.  We suppose that the macro models should have the
general form of the effective equation \cref{eq:hom_macro} and consider
three different choices of missing data in the model as described in the following.

\begin{itemize}
\item [\textit{(M1)}] \textit{Flux model.}
We choose the macro model
\begin{subequations}\label{eq:flux}
\begin{align}
  \partial_t \M &= - \M \times \bnabla \cdot \F_1 - \alpha \M \times \M \times \bnabla \cdot \F_1\,, \\
  \M(x, 0) &= \M_\mathrm{init}(x)\,,
\end{align}
\end{subequations}
where the missing information to complete the model is the flux
$\F_1$.  In case of a periodic material coefficient, $\F_1$ would
ideally be $\bnabla \M \A^H$. Then \cref{eq:flux} coincides with
\cref{eq:hom_macro}. To obtain $\F_1$, we average the product of the
material coefficient and gradient of the solution to the micro
problem in space and time using averaging kernels $K_\mu, K_\eta^0$
as explained in more detail in \Cref{sec:avg},
\begin{align*}
  \F_1 = \int_{0}^\eta \int_{-\mu}^\mu K^0
  _\eta(t) K_\mu(x) a^\varepsilon \bnabla \m^\varepsilon dx dt\,.
\end{align*}

\item[\textit{(M2)}] \textit{ Field model.}
Here the macro model is given by
\begin{align*}
  \partial_t \M &= - \M \times \F_2 - \alpha \M \times \M \times \F_2\,, \\
  \M(x, 0) &= \M_\mathrm{init}(x)\,,
\end{align*}
where $\F_2$ takes the role of the effective field.  In the periodic
case, $\F_2$ should hence approximate
$\bnabla \cdot (\bnabla \M \A^H)$.  In general, $\F_2$ is defined as
the average of the operator $\L$ applied to the solution to the
micro problem,
\begin{align*}
  \F_2 = \int_{0}^\eta \int_{-\mu}^\mu K_\eta^0(t) K_\mu(x) \bnabla \cdot (a^\varepsilon \bnabla \m^\varepsilon) dx dt\,.
\end{align*}

\item[\textit{(M3)}] \textit{ Torque model.}
The third macro model we consider is
\begin{align*}
  \partial_t \M &= - \F_3 - \alpha \M \times \F_3\,, \\
  \M(x, 0) &= \M_\mathrm{init}(x)\,,
\end{align*}
which means that for a periodic material coefficient, $\F_3$  should approximate the torque
$\M \times \bnabla \cdot (\bnabla \M \A^H)$. Here $\F_3$ is given by
\begin{align*}
  \F_3 = \int_{0}^\eta \int_{-\mu}^\mu K_\eta^0(t) K_\mu(x) \m^\varepsilon \times \bnabla \cdot (a^\varepsilon \bnabla \m^\varepsilon) dx dt\,.
\end{align*}
\end{itemize}

In the following, we prove estimates for the upscaling error in each
of the three models, (M1) - (M3), when $a^\varepsilon$ is
periodic,
\begin{align*}
  E_1 :&= \left| \F_1 - \bnabla \M_0 \A^H \right|, \\ \qquad
  E_2 :&= \left| \F_2 - \bnabla \cdot (\bnabla\M_0 \A^H) \right|, \\\qquad
  E_3 :&= \left| \F_3 - \M_0 \times (\bnabla \cdot (\bnabla \M_0 \A^H)) \right|,
\end{align*}
under the assumption that
$\partial_x^\beta \m_\mathrm{init}(0, 0) = \partial_x^\beta
\M_0(x_\mathrm{M}, t_\mathrm{M})$ for multi-indices $\beta$ with
$|\beta| \le 2$, where $(x_\mathrm{M}, t_\mathrm{M})$ is the
macro point we average around.

\section{Homogenized solution and correctors for a periodic micro
  problem} \label{sec:hom}

To be able to prove estimates for the upscaling errors $E_1$ - $E_3$
given a periodic material coefficient, we make use of the estimates
for the error between the actual and the homogenized solution to
\cref{eq:micro_prob} as well as the corresponding corrected
approximations that were derived in \cite{paper1}.

Let $\m^\varepsilon$ be the solution to \cref{eq:micro_prob} given
$a^\varepsilon(x) = a(x/\varepsilon)$. Then, according to
\cite{paper1}, the corresponding homogenized solution is $\m_0(x,t)$,
which for $0 \le t \le \eta$ satisfies
\begin{subequations} \label{eq:hom}
\begin{align}
  \partial_t \m_0 &= - \m_0 \times \bnabla \cdot (\bnabla \m_0 \A^H) - \alpha
\m_0 \times \m_0 \times \bnabla \cdot (\bnabla \m_0 \A^H) \,, \quad  x \in \Omega\,, \\
  \m_0(x, 0) &= \m_\mathrm{init}(x)\,,
\end{align}
\end{subequations}
with periodic boundary conditions and where $\m_\mathrm{init}$ is
chosen as in \cref{eq:micro_prob}.

We consider a short time interval $[0, T^\varepsilon]$, with
an $\varepsilon$-dependent final time
\begin{equation}
  \label{eq:Teps}
  T^\varepsilon := \varepsilon^\sigma T, \qquad  1 < \sigma \le 2.
\end{equation}
This is still sufficiently long time for the HMM micro problems with final time $\eta \sim \varepsilon^2$.
In \cite{paper1}, it is shown that for such a time interval, one
obtains improved approximations to the solution to
\cref{eq:micro_prob}, $\m^\varepsilon(x)$, when not only considering
$\m_0$ but a truncated asymptotic expansion
\begin{equation}\label{eq:asymptotic_expansion}
  \tilde \m^\varepsilon_J(x, t) = \m_0(x, t) + \sum_{j=1}^J \varepsilon^j \m_j\left(x, \frac{x}{\varepsilon}, t, \frac{t}{\varepsilon^2}\right), \qquad J > 0,
\end{equation}
where the correctors $\m_j(x, y, t, \tau)$, $j > 0$ satisfy linear
differential equations in the fast variables $y$ and $\tau$,
\begin{subequations}\label{eq:mj_linear}
  \begin{align*}
    \partial_\tau \m_j &= - \m_0 \times \L_{yy} \m_j - \alpha \m_0 \times \m_0 \times \L_{yy} \m_j - \f_j, \\
    \m_j(x, y, t, 0) &= 0.
  \end{align*}
\end{subequations}
Here the forcing $\f_j$ depends only on the lower order terms $\m_k$, $0 \le k < j$.
The operator $\L_{yy}$ is the vector-equivalent to $L_{yy}$, which is defined such that
for $u(x, y) \in H^{0,2}(\Omega; Y)$,
\begin{equation}
  \label{eq:Lyy}
  L_{yy} u(x, y) = \nabla_y \cdot (a(y) \nabla_y u(x, y)).
\end{equation}
As explained in \cite{paper1},
the form of the
first corrector $\m_1$ is 
\begin{equation}
  \label{eq:m1}
  \m_1(x, y, t, \tau) = \bnabla \m_0(x,t) \boldsymbol \chi(y) + \v(x, y, t, \tau)\,,
\end{equation}
where $\boldsymbol \chi$ is the solution to  \cref{eq:cell_problem}
and $\v$ is an oscillatory, decaying term that is discussed in Section \ref{sec:v}.

\subsection{Energy and error estimates}
For convenience of the reader, we here give a summary of the
estimates from \cite{paper1} that are most crucial for the
derivations in this paper. In contrast to \cite{paper1}, we here
require higher regularity of $\m_0$ for reasons of simplicity.
Otherwise we use the same assumptions.
In particular, we assume that
\begin{itemize}
\item [(A1)] the material coefficient $a \in C^\infty(\Omega)$ is
  a periodic function such that there are positive constants
  $a_\mathrm{min}, a_\mathrm{max} > 0$ satisfying
  $a_\mathrm{min} < a(x) < a_\mathrm{max}$.
\item [(A2)] the initial data function is normalized,
  $|\m_\mathrm{init}(x)| = 1$, which implies that
  $|\m^\varepsilon(x, t)| = |\m_0(x, t)| = 1$ for any $t \ge
  0$. From this property it follows that given a
  multi-index $\beta$ with $|\beta|= 1$,
  \[0 = \partial^\beta |\m_0|^2 = 2 \m_0 \cdot \partial^\beta \m_0,\]
  and thus $\m_0$ and $\bnabla \m_0$ are orthogonal.
\item [(A3)] the damping coefficient $\alpha$ is positive and it
  holds that $0 < \alpha \le 1$. Moreover, $\varepsilon = \ell/n$
  for some $n \in \mathbb{N} \gg 1$, which implies
  $0 < \varepsilon \ll 1$.

\item
  [(A4)]$\m^\varepsilon \in C^1([0, T^\varepsilon];
  H^{s+1}(\Omega))$ with $s \ge 1$ is a classical solution to
  \cref{eq:micro_prob} and there is a constant $M$ independent of
  $\varepsilon$ such that
  \[\|\bnabla \m^\varepsilon(\cdot, t)\|_{L^\infty} \le M\,, \qquad 0 \le t \le T^\varepsilon\,.\]
\item [(A5)] $\m_0 \in C^\infty(0, T; H^{\infty}(\Omega))$ is a
  classical solution to \cref{eq:hom}.
\end{itemize}

As shown in \cite{paper1}, it then holds for any $q \ge 0$ that the
$H^{q, \infty}$-norms of the first two correctors, $\m_1$ and
$\m_2$, are bounded uniformly in the fast time variable $\tau$ ,
while the norms of higher order correctors grow algebraically with
$\tau$.  Specifically, it holds for all $p, q \ge 0$ and
$0 \le t \le
T^\varepsilon$
that
\begin{align}\label{eq:mj_norm_bound}
  \|\m_j(\cdot, \cdot, t, t/ {\varepsilon^2})\|_{H^{q, p}} \le C
  \begin{cases}
    1, & j = 1, 2, \\
    \varepsilon^{(\sigma-2)(j-2)}, & j \ge 3\,,
  \end{cases}
\end{align}
where the constant $C$ depends on $T$ but is independent of
$\varepsilon$. For the approximating $\tilde \m_J^\varepsilon$ in \cref{eq:asymptotic_expansion}, it
holds for $0 \le t \le T^\varepsilon$ that
\begin{align}\label{eq:tilde_mJ_infty}
  \|\tilde \m_J^\varepsilon\|_{W^{q, \infty}} \le C \varepsilon^{\min(0, 1-q)}, \qquad q \ge 0.
\end{align}

Moreover, consider the error introduced when approximating
$\m^\varepsilon$ by $\tilde \m^\varepsilon_J$.
Under the given assumptions, it
holds for $0 \le t \le T^\varepsilon$ and $q \le s$ that
\begin{equation}
  \label{eq:error_mJbest}
  \|\m^\varepsilon(\cdot, t) - \tilde \m^\varepsilon_J(\cdot, t)\|_{H^q}
  \le C \varepsilon^{2 - q + (\sigma - 1)(J-1)}\,, \qquad J \ge 1,
\end{equation}
where the constant $C$ is independent of $\varepsilon$ but depends
on $M$ in (A4) and $T$. This estimate shows that the approximations improve
with increasing $J$ on the considered time
interval.

\subsection{The correction term \v}\label{sec:v}

In \cite{paper1}, it was shown that both $\m_1$ and the correction term $\v$, which
is part of $\m_1$ as given in \cref{eq:m1}, are orthogonal to $\m_0$,
\begin{equation}
  \label{eq:m1_v_orth}
  \v \perp \m_0, \qquad \m_1 \perp \m_0\,.
\end{equation}
Moreover, it was proved that given (A1)-(A5), there are constants
$\gamma > 0$ and $C$ independent of $\varepsilon$ such that
\begin{equation}
  \label{eq:v_norm_decay}
  \|\partial_t^k \v(\cdot, \cdot, t, \tau)\|_{H^{p, q}} \le C \varepsilon^{-\gamma \tau}, \qquad k, p, q \ge 0.
\end{equation}
For the analysis in this paper, an explicit formulation for  $\v(x, y, t, \tau)$ 
is required. To obtain such a description,
we use the linear equation 
that was derived in \cite{paper1},
\begin{subequations} \label{eq:v_eq}
  \begin{align}
      \partial_\tau \v &= - \m_0 \times \L_{yy} \v- \alpha \m_0 \times \m_0 \times \L_{yy} \v\,, \\
    \v(x, y, t, 0) &= - \bnabla \m_0(x, t) \boldsymbol \chi(y)\,.
\end{align}
\end{subequations}
We now introduce a lemma that shows a connection between
differential equations of the same type as \cref{eq:v_eq} to a
system of parabolic equations that become Schr\"odinger equations as
$\alpha \to 0$. Then we go on and use that result to derive an
explicit solution to \cref{eq:v_eq} in terms of the eigenfunctions
of the operator $-L_{yy}$.

\begin{lemma}\label{lemma:schrodinger_connection}
  Suppose $\f \in H^2(Y, \Real^3)$ and $\b \in \Real^3$ is a given constant vector with $|\b| =
  1$. Then the solution $\w$ to
    \begin{equation}
      \label{eq:lin_prob}
      \begin{split}
        \partial_t \w(y, t) &= - \b \times \L_{yy} \w(y,t) - \alpha \b \times \b \times \L_{yy} \w(y,t) , \qquad y \in Y, ~t > 0\,, \\
        \w(y, 0) &= \f(y)\,,
      \end{split}
    \end{equation}
    with periodic boundary conditions is given by
   \begin{equation}
    \label{eq:m_schrodinger_def}
    \w(y, t) := \b \b^T \f(y) + (\I - \b \b^T) \Re(\u(y, t)) + \b \times \Im(\u(y, t))\,,
  \end{equation}
  where $\u \in C^1(0, T; H^2(Y; \mathbb{C}^3))$ solves
  \begin{equation}
    \label{eq:schrodinger_type_eq}
    \begin{split}
      \partial_t \u(y, t) &= - (i - \alpha) \L_{yy} \u(y, t) , \qquad y \in Y, ~t > 0\,, \\
      \u(y, 0) &= \f(y)\,
    \end{split}
  \end{equation}
  with periodic boundary conditions.
  \end{lemma}

  \begin{proof}
    As $\u(y, 0)$ is real, it follows immediately that $\w$ given by \cref{eq:m_schrodinger_def}
    satisfies the initial condition in \cref{eq:lin_prob}.
    Moreover, since $\b$ is constant,
    \begin{align*}
      \b \times \L_{yy} \w &= \b \times [\b \b^T \L_{yy} \f + (\I - \b \b^T) \L_{yy} \Re(\u)
                         + \b \times \L_{yy} \Im(\u)] \\
        &= \b \times \Re(\L_{yy} \u) + \b \times \b \times \Im(\L_{yy} \u)\\
        &= \b \times \Re(\L_{yy} \u) - (\I - \b \b^T) \Im(\L_{yy} \u)\,,
    \end{align*}
    where we used the vector triple product identity
     for the last step. It then follows that
    \[\b \times \b \times \L_{yy} \w = \b \times \b \times \Re(\L_{yy} \u) - \b \times \Im(\L_{yy} \u)\,.\]
    It thus holds that
    \begin{align*}
       \b \times \L_{yy} \w &+ \alpha \b \times \b \times \L_{yy} \w \\
      &= \b \times [\Re(\L_{yy} \u) - \alpha \Im(\L_{yy}\u)] + \b \times \b \times [\Im(\L_{yy} \u)
        + \alpha \Re(\L_{yy} u)]\\
      &= \b \times \Im[(i - \alpha) \L_{yy} \u] - \b \times \b \times \Re[(i - \alpha) \L_{yy} \u] .
    \end{align*}
    Using \cref{eq:schrodinger_type_eq} and exploiting the facts that $\b$ is
    constant and $\f$ is independent of time, we obtain
    \begin{align*}
       \b \times \L_{yy} \w + \alpha \b \times \b \times \L_{yy} \w
      &= - \b \times \Im (\partial_t \u) + \b \times \b \times \Re(\partial_t \u) \\
      &= - \partial_t [(\I - \b \b^T) (\Re(\u)) + \b \times \Im(\u)] = - \partial_t \w\,,
    \end{align*}
    which shows that $\w$ given by \cref{eq:m_schrodinger_def} satisfies \cref{eq:lin_prob}.
  \end{proof}

  In the following, let $\phi_j(y), \omega_j$ be the eigenfunctions
  and eigenvalues of the operator $- L_{yy}$, where $L_{yy}$ is
  given by \cref{eq:Lyy}, on $Y$ with periodic boundary conditions,
\[- L_{yy} \phi_j(y) = \omega_j \phi_j(y) \quad \text{on }
  Y\,.\]
As $- L_{yy}$ is a periodic elliptic operator
it
holds according to standard theory that its
eigenvalues are strictly positive and bounded away from zero except
for the first eigenvalue, $\omega_0$, which is zero \cite{Krein:48},
\[\omega_0 = 0, \qquad 0 < \omega_j, \quad\text{for}\quad  j  > 0.\]
Moreover, the eigenfunctions $\phi_j$ form an orthonormal basis for $L^2(Y)$.  In
particular, the first eigenfunction, corresponding to $\omega_0$, is
the constant function $\phi_0 \equiv 1$. The eigenfunctions can be
chosen to be real, which they are assumed to be in the following.
We then obtain the following expression for the correction term $\v$.
\begin{lemma}\label{lemma:v}
  Let $\f_\v(x,t) := \bnabla \m_0(x,t) - i \m_0(x,t) \times \bnabla \m_0(x,t)$ and
\begin{align}\label{eq:psi23}
  \Psi(y, \tau) :=  \sum_{j=1}^\infty \boldsymbol \chi_j e^{(- \alpha+i) \omega_j \tau} \phi_j(y)\,,
\end{align}
where $\phi_j$ and $\omega_j$, $0 \le j$,  are the eigenfunctions and eigenvalues
of $-L_{yy}$ and $\boldsymbol \chi_j$ are expansion coefficients such that $\boldsymbol \chi(y) = \sum_j \boldsymbol \chi_j \phi_j(y)$. Then
\[\v(x, y, t, \tau) = - \mathrm{Re}(\f_\v(x,t) \Psi(y, \tau))\]
solves \cref{eq:v_eq}\,.
\end{lemma}

\begin{proof}
  Since $\v$ satisfies a linear differential equation in the fast
  variables, we write $\v(y, \tau)$ and suppress the dependence on
  the slow variables, $x$ and $t$, in the notation throughout this
  proof.  Note also that with respect to the fast variables only,
  $\m_0$ and $\bnabla \m_0$ are constant.

  By \Cref{lemma:schrodinger_connection} and using the fact that by (A2),
  $\m_0$ and $\bnabla \m_0$  are orthogonal to each other, we find
  that the solution to \cref{eq:v_eq} is
  \begin{align}\label{eq:v_lemma21}
    \v(y, \tau) &= (\I - \m_0 \m_0^T) \Re(\u(y, \tau)) + \m_0 \times \Im(\u(y, \tau)),
  \end{align}
  where $\u$ is the solution to
  \begin{equation}\label{eq:schrodinger_v}
    \begin{split}
      \partial_\tau \u(y, \tau) &= - (i - \alpha) \L \u(y, \tau)\,, \qquad y \in Y, \tau > 0,\\
      \u(y, 0) &= - \bnabla \m_0 \boldsymbol \chi(y)\,.
    \end{split}
  \end{equation}
  As \cref{eq:schrodinger_v} is a system of three decoupled
  equations, we can consider each equation separately and solve it
  in terms of the eigenfunctions of $-L_{yy}$. Let $u(y, \tau)$
  denote the first component in $\u$. Then we can define $u_j(\tau)$
  such that
  \[u(y, \tau) = \sum_{j = 0}^\infty u_j(\tau) \phi_j(y)
  \,.\]
  Note that $\boldsymbol \chi_0 = 0$ as $\boldsymbol \chi$ has zero average by definition.
  By \cref{eq:schrodinger_v}
  and the orthogonality of the eigenfunctions, we deduce that
  \[\partial_\tau u_j(\tau) = (i - \alpha) \omega_j u_j(\tau)\,, \qquad u_j(0) = - \nabla m_0^{(1)} \cdot \boldsymbol \chi_j,\]
  and consequently,
  \[u_j(\tau) = - \nabla m_0^{(1)} \cdot \boldsymbol \chi_j e^{(i - \alpha)
    \omega_j \tau}\,.\]
  For the second and third components in $\u$, we obtain the same
  result but with initial conditions involving $\nabla m_0^{(2)}$ and
  $\nabla m_0^{(3)}$, respectively. Hence, in total it holds that
  \[\u(y, \tau) = - \bnabla \m_0 \sum_{j=1}^\infty  \boldsymbol \chi_j e^{(i - \alpha) \omega_j \tau} \phi_j(y) = - \bnabla \m_0 \Psi(y, \tau)\,,\]
  where $\Psi(y, \tau)$ is defined as in \cref{eq:psi23}. Putting
  this explicit expression for $\u$ into \cref{eq:v_lemma21}, then results in
  \begin{align*}
    \v(y, \tau) 
&= - \bnabla \m_0 \Re(\Psi(y, \tau)) - \m_0 \times \bnabla \m_0 \Im(\Psi(y, \tau)) \\ &=
- \Re(\bnabla \m_0 \Psi(y, \tau) - i \m_0 \times \bnabla \m_0 \Psi(y, \tau)).
  \end{align*}
  This completes the proof.
\end{proof}
To gain a more intuitive understanding, note that $\v(y, \tau)$ can
also be written as
\begin{align}\label{eq:v_sin_cos}
  \v(y, \tau) &= - \sum_{j=1}^\infty \boldsymbol \chi_j \phi_j(y) e^{-\alpha \omega_j \tau} [\cos(\omega_j \tau) \bnabla \m_0  + \sin(\omega_j \tau) \m_0 \times \bnabla \m_0]\,.
\end{align}
As $\bnabla \m_0$ and $\m_0 \times \bnabla \m_0$ are orthogonal to
$\m_0$ and each other, this clearly shows that $\v$ lies in the
subspace orthogonal to $\m_0$ and can be written in terms of two
orthogonal vectors spanning this subspace multiplied by coefficients
that oscillate with $\tau$. For $\alpha > 0$, all the components of
$\v$ are damped away with increasing $\tau$, with stronger damping
for higher modes. 
Note that the sum in \cref{eq:v_sin_cos} starts from $j = 1$. There
is no contribution from the constant mode, indicating that $\v$ has
zero average.

\section{Averaging}\label{sec:avg}

In order to get a good approximation of the missing quantity for the
macro model in our HMM scheme, it is crucial to have efficient
averaging techniques that allow us to control how fast the averaged
micro model data converges to the required effective quantity. To
achieve this, one can use smooth, compactly supported
averaging kernels as introduced in \cite{stiff}, \cite{doghonay1}.

\begin{definition}[\cite{doghonay1}]\label{def:kernel}
A function $K$ is in the space of smoothing kernels $\mathbb{K}^{p, q}$ if
\begin{enumerate}
\item $K^{(q+1)} \in BV(\Real)$ and $K$ has compact support in $[-1, 1]$, $K \in C_c^{q}([-1, 1])$.
\item $K$ has $p$ vanishing moments,
  \[\int_{-1}^1 K(x) x^r dx =
  \begin{cases}
    1\,, & r = 0\,,\\
    0\,, & 1 \le r \le p\,.
  \end{cases}\]
\end{enumerate}
\end{definition}

Typically, we do not want to average over $[-1, 1]$ but over small
boxes of size proportional to $\varepsilon$ or $\varepsilon^2$. For
this purpose, let $K_\mu(x)$ denote a scaled version of $K(x)$,
\begin{align*}
  K_\mu(x) = \frac{1}{\mu} K\left(\frac{x}{\mu}\right)\,.
\end{align*}
Moreover, when considering problems in $d$ space dimensions with $d > 1$, $K_\mu(x)$ is to be understood as
\[K_\mu(x) = K_\mu(x_1) \cdots K_\mu(x_d)\,.\]
Note that as $K$ has compact support and $K^{(q+1)} \in BV(\Real)$,
it holds that $K \in W^{1, q+1}(\Real)$ and $K \in L^2(\Real)$.

\subsection{Kernels $K^0$}

Often, the averaging kernels used for HMM are chosen to be symmetric
around zero and have nonzero-values almost everywhere in $[-1, 1]$.
However, for our application it is
advantageous to do time averaging such that we obtain an
approximation for the effective quantity at time $t = 0$ based only
on the values of the microscopic solution for $t \ge 0$.
As the subsequent proofs require kernels $K \in \mathbb{K}^{p, q}$,
we therefore show that $\mathbb{K}^{p, q}$ contains a subspace
$\mathbb{K}_0^{p, q}$ such that $K^0(t) = 0$ for $t \le 0$ when
$K^0 \in \mathbb{K}_0^{p, q}$.
To construct such kernels, consider the ansatz
\begin{equation}
  \label{eq:k_0}
  K^0(t) =
\begin{cases}
  t^{q+1} (1-t)^{q+1} P(t)\,, & 0 < t < 1 \\
  0\,,& otherwise\,,
\end{cases}
\end{equation}
where $P$ is a polynomial in $\mathbb{P}^p$, the space of of
polynomials of degree $p$,
\[P(t) = c_0 + c_1 t + ... + c_p t^p\,.\]
As explained in \cite{henrik2}, it is beneficial to choose this type
of ansatz since it typically results in better numerical stability
compared to an approach where the coefficients of $K^0$ are computed
directly.

One can easily see that due to the term $t^{q+1} (1-t)^{q+1}$, the first $q$
derivatives of $K^0$ as given by \cref{eq:k_0} vanish at zero and
one, which together with continuity implies that the first
requirement in \cref{def:kernel} is satisfied.

To show that there indeed exists a unique polynomial $P$ in
$\mathbb{P}^p$ such that $K^0$ as given in \cref{eq:k_0} also
satisfies the second requirement in \Cref{def:kernel} and hence is
in $\mathbb{K}^{p, q}$, we define the weighted inner product
$ \langle \cdot, \cdot\rangle_w$ by
\begin{align*}
\langle u, v \rangle_w := \int_{0}^1 u(t) v(t) t^{q+1} (1-t)^{q+1} dt, \qquad \|u\|_w := \langle u, u\rangle_w\,.
\end{align*}
This allows us to rewrite the 
second condition in \Cref{def:kernel} as
\begin{equation}\label{eq:k_cond_inner}
  \begin{split}
  \langle P, 1 \rangle_w &= 1, \\
  \langle P, t^k \rangle_w &= 0\,, \quad 1 \le k \le p\,.
    \end{split}
\end{equation}
Let now $\phi_j, $ $j = 0, \ldots , p$ be orthogonal polynomials with
respect to $\langle \cdot, \cdot \rangle_w$, satisfying the
recurrence formula
\begin{align*}
  \phi_0 = 1\,, \quad \phi_1 = (x - \alpha_0) \phi_0 \,, \quad \phi_{k+1} = (x-\alpha_k) \phi_k - \beta_k \phi_{k-1}\,,
\end{align*}
where $\alpha_k = \frac{\langle \phi_k, x \phi_k\rangle_w}{\|\phi_k\|_w^2}$ and $\beta_k =
\frac{\|\phi_k\|_w^2}{\|\phi_{k-1}\|_w^2}$.
Then it holds that $\phi_j \in \mathbb{P}^j$
and together the
$\phi_j, j = 0, \ldots, p$ form an orthogonal basis for $\mathbb{P}^p$. We can
hence expand
\begin{align}\label{eq:k_poly_exp}
  P(t) = \sum_{j=1}^p p_j \phi_j(t)\,, \qquad t^k = \sum_{j=0}^p c_{jk} \phi_j(t)\,,
\end{align}
where the coefficients $c_{jk}$ are uniquely determined
\cite{powell}. In particular, $c_{jj} = 1$ and $c_{jk} = 0$ for
$k > j$.  Expressing the inner product in \cref{eq:k_cond_inner} in
terms of the expansions \cref{eq:k_poly_exp} yields
\begin{align*}
  \langle P, t^k \rangle_w = \sum_{j=1}^k p_j c_{jk} \|\phi_j\|_w^2\,,
\end{align*}
which implies that \cref{eq:k_cond_inner} is satisfied when the
coefficients $p_j$ are the solution to 
\begin{align*}
  \begin{bmatrix}
    \|\phi_0\|^2_w & 0 & 0 & \dots & 0 \\
    c_{01} \|\phi_0\|^2_w & \|\phi_1\|^2_w & 0 & \dots & 0\\
    c_{02} \|\phi_0\|^2_w & c_{12} \|\phi_1\|^2_w & \|\phi_2\|^2_w & \dots & 0\\
    \vdots & \vdots & & \ddots & \\
    c_{0p} \|\phi_0\|^2_w & c_{1p} \|\phi_1\|^2_w & \dots & & \|\phi_p\|^2
  \end{bmatrix}
  \begin{bmatrix}
    p_0 \\ p_1 \\ p_2 \\ \vdots \\ p_p
  \end{bmatrix}
=
  \begin{bmatrix}
    1 \\ 0 \\ 0 \\ \vdots \\ 0
  \end{bmatrix}\,.
\end{align*}
Since the matrix here is triangular with strictly positive diagonal
elements, the system has a unique solution, which proves that there
exists a unique polynomial $P$ of degree at most $p$ such that $K^0 \in \mathbb{K}_0^{q, p} \subset \mathbb{K}^{q, p}$.

\begin{remark}
In practice, there is a quicker way to determine the coefficients $c_j$ of the polynomial $P(t)$.
 Let $I_j = \int_0^1 t^{q +1+j} (1-t)^{q+1} dt$, then
it has to hold that the vector containing the coefficients $c_j$ solves the linear system
\begin{align*}
  \begin{bmatrix}
    I_0 & I_1 & ... & I_p \\
    I_1 & I_2 & ... & I_{p+1} \\
    \vdots & & \ddots & \\
    I_r & I_{p+1} & \dots & I_{2p}
  \end{bmatrix}
\begin{bmatrix}
c_0 \\ c_1 \\ \vdots \\ c_p
\end{bmatrix} =
\begin{bmatrix}
1 \\ 0 \\ \vdots \\ 0
\end{bmatrix}\,.
\end{align*}
\end{remark}

\subsection{Averaging in space}

The following lemma from \cite{doghonay1} gives a precise
convergence rate in terms of $\varepsilon/\eta$ when averaging a
purely periodic function $f$ with a kernel over a one-dimensional
interval. By choosing a kernel with high regularity, one can achieve
very fast convergence to the corresponding average.

\begin{lemma}[\cite{doghonay1}]\label{lemma:doghonay1}
Let $f: \Real \to \Real$ be a 1-periodic continuous
  function, and let $K \in \mathbb{K}^{p, q}$. Then, with $\bar f = \int_0^1 f(y) dy$,
  \begin{align*}
    \left|\int_\Real K_\mu(x) f\left(\frac{x}{\varepsilon}\right) dx - \bar f\right| \le
C |f|_\infty \left(\frac{\varepsilon}{\mu}\right)^{q+2}
  \end{align*}
  and when $r \in \mathbb{Z}^+$,
  \begin{align*}
    \left|\int_\Real K_\mu(x) x^r f\left(\frac{x}{\varepsilon}\right) dx\right| \le
C
    \begin{cases}
      |f|_\infty \left(\frac{\varepsilon}{\mu}\right)^{q+2} \mu^r\,, & 1 \le r \le p\,, \\
      |f|_\infty \left(\frac{\varepsilon}{\mu}\right)^{q+2} \mu^r + |\bar f|\mu^r\,, & r > p\,,
    \end{cases}
  \end{align*}
  where the constant $C$ is independent of $\varepsilon$, $\mu$,
  $f$ or $x$ but may depend on $K, p, q$ and $r$.
\end{lemma}
In \cite{doghonay1}, this lemma is proved for a continuous $f$
since all derivatives involved in the proof are treated in the classical
sense. However, when the derivatives are seen in a weak sense, the
lemma also applies to $f \in L^\infty$, as explained in \cite{doghonay2}.

In \cite{doghonay2}, an averaging lemma for functions $f(x, y)$ that only are
1-periodic in the second variable is derived.
In the following, we give a variation of that lemma which is adapted
for Bochner-Sobolev spaces and higher dimensions.
 \begin{lemma}\label{lemma:avg_space}
   Let $Y = [0, 1]^d$ and $\Omega_\mu = [-\mu, \mu]^d$ for
   $d \in \mathbb{N}$. Suppose $f(x, y)$ is 1-periodic in $y$ and
   $\partial_x^\beta f \in L^\infty(\Omega_\mu; L^\infty(Y))$ for
   $0 \le |\beta| \le p+1$ and assume that
   $K \in \mathbb{K}^{p, q}$. Then, with
   $\bar f(x) := \int_{Y} f(x, y) dy$,
    \begin{align*}
      \left| \int_{\Real^d} K_\mu(x) f(x, x/\varepsilon) dx - \bar f(0) \right| \le C  \sup_{y \in Y} \|f(\cdot, y)\|_{W^{p+1, \infty}(\Omega_\mu)}
      \left(\left(\frac{\varepsilon}{\mu}\right)^{q+2} + \mu^{p+1}\right)\,,
    \end{align*}
    where the constant $C$ does not depend on $\mu$ or $f$ but may depend on $K, p$ and $q$.
  \end{lemma}

  \begin{proof}
    We first assume that
    $\partial_x^\beta f \in C(\Omega; L^\infty(Y))$ for
    $0 \le |\beta| \le p+1$. Then we obtain via Taylor
    expansion of $f(x, x/\varepsilon)$ that
  \begin{align*}
    \int_{\Real^d} K_\mu(x) &f\left(x, {x}/{\varepsilon}\right) dx =
    \int_{\Omega_\mu} K_\mu(x) f\left(0, {x}/{\varepsilon}\right) dx \\&+ \sum_{1 \le |\beta| \le p} \frac{1}{\beta!}\int_{\Omega_\mu} K_\mu(x)\partial_x^\beta f\left(0, {x}/{\varepsilon}\right)  x^\beta dx
    +  \sum_{|\beta| = p+1} \int_{\Omega_\mu} K_\mu(x) R_\beta(x) x^\beta dx
    \\&=: I + II + III
\,,
  \end{align*}
  where $R_\beta(x)$ is the remainder in integral form,
  \[R_\beta(x) = \frac{|\beta|}{\beta!} \int_0^1 (1-z)^{p} \partial_x^{\beta} f\left(z x, {x}/{\varepsilon}\right) dz\,.\]

  The terms $I$ and $II$ can be bounded using
  \Cref{lemma:doghonay1}. We consider one coordinate direction at a
  time. For this purpose, assume that we have a multi-index
  $\beta = [\beta_1, ..., \beta_d]$, coordinates $x = (x_1, ... x_d)$  and let
  \begin{align*}
    g_0(y_1, ..., y_d) &:= \partial_x^\beta f(0, y) = \partial_x^\beta f(0, ..., 0 , y_1, ..., y_d), \\
    g_{n} (y_{n+1}, ..., y_d) &:= \int_{-\mu}^\mu K_{\mu}(x_{n})  g_{n-1}(x_{n}/\varepsilon, y_{n+1},..., y_d) x_{n}^{\beta_n} dx_{n},  \qquad 1 \le n \le d-1\\
    g_d &:= \int_{-\mu}^\mu K_{\mu}(x_{d})  g_{d-1}(x_{d}/\varepsilon) x_{d}^{\beta_d} dx_{d}
          = \int_{\Omega_\mu} \K_\mu(x) \partial_x^\beta f(0, x/\varepsilon) x^\beta dx
          ,
  \end{align*}
   and
  \begin{align*}
    h_n(x_n) := \int_{0}^1 \cdots \int_0^1  g_{n-1}(x_n, y_{n+1}, ..., y_d) dy_{n+1} \cdots dy_d.
  \end{align*}
  Note first that due to the
  fact that $\mu < 1$, we obtain by iterative application of
  \Cref{lemma:doghonay1} that
  \begin{align}\label{eq:sup_gn}
    \sup_{y \in Y} |g_n(y_{n+1}, ..., y_d)|
    &= \sup_{y \in Y}\left| \int_{-\mu}^\mu K_\mu(x_n) g_{n-1} (x_n/\varepsilon, y_{n+1}, ..., y_d) x_n^{\beta_n} dx_n\right| \nonumber \\
    &\le C \sup_{y \in Y} | g_{n-1}(y_n, ..., y_d)| \left(\left(\varepsilon/\mu\right)^{q+2} + \delta_{\beta_n = 0}\right) \\
    &\le C \sup_{y \in Y} |\partial_x^{\beta} f(0, y)| \prod_{j=1}^n \left(\left(\varepsilon/\mu\right)^{q+2} + \delta_{\beta_j = 0}\right) \nonumber,
  \end{align}
  where $\delta_{\beta_j = 0}$ indicates that there is a term of
  order one when $\beta_j = 0$, an upper bound for the average in coordinate direction
  $j$.

  In case of $I$, we have $|\beta| = 0$. An application of
  \Cref{lemma:doghonay1} then yields that for $1 \le j \le d$,
  \begin{align*}
    \left| \bar g_j - \bar g_{j-1}\right|
    &= \left| \int_{-\mu}^\mu K_\mu(x_j) h_{j}(x_j/\varepsilon) dx_j
      - \int_0^1 h_j(x_j) dx_j \right| \le C \sup_{y_j \in [0, 1]} |h_j(y_j)| \left(\frac{\varepsilon}{\mu}\right)^{q+2},
  \end{align*}
  and as a consequence of \cref{eq:sup_gn}, it holds that
  \begin{align*}
    \sup_{y_j \in [0, 1]} |h_j(x_j)| \le \sup_{y_j, ..., y_d \in [0, 1]} |g_{j-1}(y_n, ..., y_d)| \le
    C \sup_{y \in Y} |f(0, y)|.
  \end{align*}
  Using the fact that $I = g_d = \bar g_d$ and $\bar f(0) = \bar g_0$ we hence obtain
  \begin{align}\label{eq:space_I}
    |I - \bar f(0)| \le |I - \bar g_{d-1}| + \sum_{j=1}^{d-1} |\bar g_{j} - \bar g_{j-1}|
    \le C \sup_{y \in Y} |f(0, y)| \left(\frac{\varepsilon}{\mu}\right)^{q+2}.
  \end{align}

  To estimate the integrals in $II$, consider $1 \le |\beta| \le
  p$. It then follows by 
  \cref{eq:sup_gn} that
  \begin{align*}
    \left| \int_{\Omega_\mu} K_\mu(x) \partial_x^\beta f(0, x/\varepsilon) x^\beta dx \right|
    &= |g_d|
    \le C \sup_{y \in Y} |\partial_x^\beta f(0, y)| \left(\frac{\varepsilon}{\mu}\right)^{q+2},
  \end{align*}
  where the last step follows since we know that $|\beta| > 0$,
  there is at least one direction $j$ such that $\beta_j > 0$. Consequently, we obtain
  \begin{align}\label{eq:space_II}
    |II| \le  C \max_{1 \le |\beta| \le p}  \sup_{y \in Y} |\partial_x^\beta f(0, y)| \left(\frac{\varepsilon}{\mu}\right)^{q+2} .
  \end{align}

  To bound $III$ we use 
  the fact that
  \begin{align*}
    \sup_{x \in \Omega_\mu} |R_\beta(x)|
    &\le  C \frac{|\beta|}{\beta!} \sup_{x \in [-1,1]^d} \left| \int_0^1 (1-z)^{p} \partial^\beta_x f\left(z \mu x, \frac{\mu x}{\varepsilon}\right) dz \right|
 \\ &\le C  \sup_{x \in \Omega_\mu} \sup_{y \in Y} |\partial^\beta f(x, y)|.
\end{align*}
Thus we can bound the integrals in the third term above, $III$, as follows,
  \begin{align*}
    \left| \int_{\Omega_\mu} K_\mu(x) R_\beta(x) x^\beta dx \right|
    \le \sup_{x \in \Omega_\mu}  |R_\beta(x) x^\beta| \|K\|_{L^1}
    \le C  \sup_{x \in \Omega_\mu} \sup_{y \in Y} |\partial^\beta f(x, y)| \mu^{|\beta|}.
  \end{align*}
  Therefore,
  \begin{align}\label{eq:space_III}
     |III| \le C \max_{|\beta| = p+1} \sup_{x \in \Omega_\mu} \sup_{y \in Y} |\partial^\beta f(x, y)| \mu^{p+1}.
  \end{align}
  Combining the estimates \cref{eq:space_I}, \cref{eq:space_II} and
  \cref{eq:space_III} then yields the estimate in the lemma for
  functions with $\partial_x^\beta f \in C(\Omega, L^\infty(Y))$.

  If we instead have that
  $\partial_x^\beta f \in L^\infty (\Omega_\mu; L^\infty(Y))$ for
  $0 \le |\beta| \le p+1$, we can approximate them by smooth
  functions such that the above still holds.
\end{proof}

\subsection{Averaging in space and slow time}

For a vector-valued function $\w(x, t)$, let
\begin{align}\label{eq:bar_k_def}
   \bar \K_{\mu, \eta} \w := \int_{\Omega_\mu} \int_{0}^\eta
  K_\mu(x) K^0_\eta(t) \w(x, t) dt dx\,,
\end{align}
where $K \in \mathbb{K}^{p_x, q_x}$ and
$K^0 \in \mathbb{K}^{p_t, q_t}_0$ are given kernels that are scaled
by parameters $\mu$ and $\eta$, respectively.

\begin{lemma}\label{lemma:k_boundedness}
  With $\bar \K_{\mu, \eta}$ given in \cref{eq:bar_k_def}, it holds
  for $\u \in L^\infty(0, \eta; L^2(\Omega))$ that
  \begin{align}\label{eq:k_l2_bound}
    \left|\bar \K_{\mu, \eta} \u\right| \le \frac{C}{\mu^{d/2}} \sup_{0 \le t \le \eta} \|\u(\cdot, t)\|_{L^2}.
  \end{align}
  Moreover, if $\u \in L^\infty(0, \eta; L^\infty(\Omega))$, then
  \begin{align}\label{eq:k_inf_bound}
    \left|\bar \K_{\mu, \eta} \u\right| \le C \sup_{0 \le t \le \eta} \|\u(\cdot, t)\|_{L^\infty}.
  \end{align}
\end{lemma}
\begin{proof}
 By the Cauchy-Schwarz inequality, it follows that
 that
 \begin{align*}
   \left| \int_{\Omega_\mu} K_\mu(x) \u(x,t) dx \right|
   &\le \left(\int_{\Omega_\mu} \frac{1}{\mu^{2d}} \left|K\left(\frac{x}{\mu}\right)\right|^2 dx
     \int_{\Omega_\mu} \left|\u(x, t)\right|^2dx \right)^{1/2}  \\
   &= \left( \frac{1}{\mu^d}  \int_{[-1, 1]^d} \left|K(x)\right|^2 dx  \int_{[-1, 1]^d} |\u|^2 dx \right)^{1/2} \\
   &\le  \frac{C}{\mu^{d/2}} \|\u(\cdot, t)\|_{L^2}\,,
 \end{align*}
 hence it holds that
 \begin{align*}
   |\bar \K_{\mu, \eta} \u|
   \le \frac{1}{\sqrt{\mu^d}} C  \left\|\int_0^\eta K^0_\eta \u(\cdot, t) dt\right\|_{L^2}
   \le \frac{1}{\sqrt{\mu^d}} C  \|K^0\|_{L^1} \sup_{0 \le t \le \eta}  \left\|\u(\cdot, t)\right\|_{L^2}
   \,,
 \end{align*}
 which shows the first result in the lemma. Furthermore, it holds
 that
 \begin{align*}
   |\bar \K_{\mu, \eta} \u|
   &\le \sup_{0 \le t \le \eta} \|\u(\cdot, t)\|_{L^\infty(\Omega_\mu)}  \int_{\Omega_\mu} \int_0^\eta |K_\mu| |K^0_\eta| dt dx  \\
     &\le \|K\|_{L^1} \|K^0\|_{L^1} \sup_{0 \le t \le \eta}  \left\|\u(\cdot, t)\right\|_{L^\infty(\Omega_\mu)}.
 \end{align*}
 This completes the proof.
\end{proof}

Next, we prove a general lemma that holds for the
averaging of sufficiently regular functions that change only slowly
in time but contain both slow and fast variations in space. These
fast spatial oscillations have to be representable by a periodic function
that multiplies a function only depending on the slow variables.

\begin{lemma}\label{lemma:avg1}
  Consider averaging kernels $K(x)$ for space and $K^0(t)$ for time
  such that $K \in \mathbb{K}^{p_x, q_x}$ and
  $K^0 \in \mathbb{K}^{q_t, p_t}_0$ and assume that
  $\varepsilon < \mu < 1$ and $\varepsilon^2 < \eta < 1$.  Let $g$
  be a $1$-periodic function such that $g \in L^\infty(Y) $ and
  $\bar g = \int_{[0, 1]^d} g(x) dx$.  Suppose that for
  $0 \le k \le p_t+1$ and $0 \le |\beta| \le p_x + 1$,
  $\partial_t^k \partial_x^\beta f(x, t) \in C(0, \eta;
  L^\infty(\Omega))$.  Then
  \begin{align*}
        & \left| \bar \K_{\mu, \eta} \left( f(x, t) g(x/\varepsilon)\right) - f(0, 0) \bar g \right|\\
&\le  C  \max_{k \le p_t +1} \sup_{0 \le t \le \eta} \|\partial_t^k f(\cdot, t)\|_{W^{p_x+1, \infty}(\Omega_\mu)} \|g\|_{L^\infty}  \left(\left(\frac{\varepsilon}{\mu}\right)^{q_x +2} + \mu^{p_x + 1}  + \eta^{p_t + 1} \right)\,.
  \end{align*}
\end{lemma}

\begin{proof}
Consider first averaging in time only. As an immediate consequence of
\Cref{def:kernel} and the fact that $K^0(t)$ is zero for $t < 0$, it holds that
  \begin{align*}
    \int_{0}^\eta K^0_\eta(t) f(x, 0) dt
    &= f(x, 0) \int_{-1}^1 K^0(t) dt = f(x, 0) \,,
      \\
      \int_{0}^\eta K^0_\eta(t) \partial_t^j f(x, 0) t^j dt
      &= \eta^j \partial_t^j f(x, 0) \int_{-1}^1 K^0(t)  t^j dt
      = 0, \qquad 1 \le j \le p_t.
  \end{align*}
 Hence, when Taylor-expanding $f(x, t)$ in time around zero, we obtain

\begin{align*}
   \int_{0}^\eta K^0_\eta(t) f(x, t) dt = f(x, 0) + \int_{0}^\eta K_\eta^0(t) {R}_{p_t+1}(x, t) dt\,,
\end{align*}
where the remainder term is
\[{R}_{p_t+1}(x, t) := \frac{1}{{p_t}!} \int_{0}^t (t - z)^{{p_t}} \partial_t^{{p_t}+1} f (x, z) dz\,.\]
This representation of the time
averaging integral can then be used to obtain a bound on the
considered averaging error that consists of two parts,
\begin{align*}
  &\left| \bar \K_{\mu, \eta} \left( f(x, t) g(x/\varepsilon) \right)
  -  f(0, 0) \bar g \right| \\
&\qquad\le  \left| \int_{\Omega_\mu} K_\mu(x) f (x, 0) g(x/\varepsilon) dx - f(0, 0) \bar g \right|
  +  \left| \bar \K_{\mu, \eta} \left( {R}_{p_t}(x, t)  g(x/\varepsilon) \right)\right|
 =: I + II\,.
\end{align*}
The first part here, $I$, corresponds to averaging in space of
a time-independent function with slow and fast, periodic variations in
space. Application of  \Cref{lemma:avg_space} then yields
\begin{align*}
  |I| &\le C \sup_{y \in Y} \|f(\cdot, 0) g(y)\|_{W^{p_x+1, \infty}(\Omega_\mu)}
        \left(\left(\frac{\varepsilon}{\mu}\right)^{q_x +2} + \mu^{p_x + 1} \right) \,.
\end{align*}

To bound the second part, $II$, note that the 
remainder integral from time integration can be rescaled to $[0, 1]$
and then be bounded in terms of $C_f$ and $K^0(t)$,
\begin{align*}
 \int_{0}^\eta K_\eta^0 {R}(x, t) dt 
  &\le \frac{1}{{p_t}!} \int_{0}^1 |K^0| \int_0^{\eta t} |(\eta t - z)|^{{p_t}} |\partial_t^{{p_t}+1} f(x, z)| dz dt \\
 &\le \frac{\eta^{{p_t}+1}}{{p_t}!} \sup_{0 \le t \le \eta} |\partial_t^{{p_t}+1} f (x, t)| \|K^0\|_{L^1}.
\end{align*}
Hence, we can bound the integral in $II$ by
 \begin{align*}
|II| \le
   C \sup_{0 \le t \le \eta} \sup_{x \in \Omega_\mu}   |\partial^{p_t+1} f(x, t)| \|g\|_{L^\infty}
   \eta^{p_t + 1},
 \end{align*}
 where the constant $C$ depends on $K, K^0$ and $p_t$ but is independent of $\varepsilon$, $\eta$ and $\mu$.
 Together with the estimate for $|I|$, this shows result in the lemma.
 \end{proof}

\subsection{Averaging involving temporal oscillations}

For expressions involving the correction term
$\v\left(x, x/\varepsilon, t, t/\varepsilon^2\right)$, a special averaging lemma that
exploits the structure of $\v$ as given in \Cref{lemma:v} is
necessary in order to get error estimates in \Cref{sec:avg_ll}.
We therefore proceed to  derive a lemma for time averaging with a
kernel for functions of the form
\begin{equation}
  \label{eq:w}
  \int_{-\mu}^\mu
\partial_x^{\beta_1} f(x/\varepsilon) \partial_x^{\beta_2} K_\mu(x) u(x, t) \Psi\left(\frac{x}{\varepsilon},
  \frac{t}{\varepsilon^2}\right) dx\,,
\end{equation}
where $u(x, t)$ is a function that only varies slowly, $f(y)$ is a
1-periodic function and $\Psi(x, t)$ is defined as in
\Cref{lemma:v}. Moreover, $\beta_1$ and $\beta_2$ are given
multi-indices.  We obtain the following result.
\begin{lemma}\label{lemma:avg_v_base}
  Consider averaging kernels $K \in \mathbb{K}^{q_x, p_x}$ and
  $K^0 \in \mathbb{K}^{q_t, p_t}_0$ and assume that
  $\varepsilon < \mu < 1$ and $\varepsilon^2 < \eta < 1$.  Let
  $u(x, t)$ be a complex-valued function such that
  $\partial_t^r u \in C(0, \eta; L^\infty(\Omega))$ for
  $0 \le r \le q_t +1$.  Moreover, consider multi-indices $\beta_1$
  and $\beta_2$ such that $|\beta_2| \le q_x$.  Suppose $f$ is a
  1-periodic function such that
  $\partial_x^{\beta_1}f \in L^\infty(Y)$, and that $\Psi$ is given
  by \cref{eq:psi}. Then
\begin{align*}
  &\left| \int_{0}^\eta   K^0_\eta(t) \int_{\Omega_\mu}  \left(\partial_x^{\beta_1} f(x/\varepsilon)\right) \left(\partial_x^{\beta_2} K_\mu(x)\right)  u(x, t) \Psi\left(\frac{x}{\varepsilon}, \frac{t}{\varepsilon^2}\right) dx dt \right|
  \\&\qquad\qquad  \le C   \frac{1}{\mu^{|\beta_2|} \varepsilon^{|\beta_1|}}
      \max_{\rho \le q_t+1} \sup_{t \in [0, \eta]} \|\partial_t^\rho u(\cdot, t)\|_{L^\infty}
      \|\partial_y^{\beta_1} f\|_{L^\infty}
      \left(\frac{\varepsilon^2}{\eta}\right)^{q_t+1}\,,
\end{align*}
where the constant C  depends on $ \|K^0\|_{W^{1, q_t+1}}$ 
and $\|\boldsymbol \chi\|_{L^2}$ but
is
independent of $\varepsilon, \mu, \eta$ and $t$.
\end{lemma}

\begin{proof}
  To simplify notation in the following, we introduce $\psi_j(t)$ such that
\begin{equation}
  \label{eq:psi}
  \Psi(x, t) := \sum_{j=1}^\infty \boldsymbol \chi_j \phi_j(x) \psi_j(t)\,, \qquad \text{where}\quad \psi_j(t) := e^{(-\alpha + i)\omega_j t}\,.
\end{equation}
As in \Cref{lemma:v}, $\phi_j$ and $\omega_j$, $j \ge 0$ are the
eigenfunctions and corresponding eigenvalues of the operator
$-L_{yy}$ and $\boldsymbol \chi_j$ are the expansion coefficients one obtains
when expressing $\boldsymbol \chi$ in terms of the eigenfunction basis
$\phi_j$.
 Note that the time derivatives of $\psi_j$ can be
  expressed in terms of the original function times a constant,
\begin{align*}
  \frac{d}{dt} \psi_j\left(\frac{\eta t}{\varepsilon^2}\right) = - \frac{1}{c_j} \psi_j\left(\frac{\eta t}{\varepsilon^2}\right)\,, \quad \text{with} \quad c_j := \frac{\varepsilon^2}{(\alpha - i)\omega_j \eta}\,.
\end{align*}
Repeated  application of integration by parts thus yields
\begin{align}\label{eq:cq_ibp}
  \int_{0}^\eta K^0_\eta(t) u(x, t) \psi_j(t/\varepsilon^2) dt
&= - c_j \int_{0}^1 K^0(t) u(x, \eta t) \frac{d}{dt} \psi_j\left(\frac{\eta t}{\varepsilon^2}\right) dt \nonumber \\
&= c_j^{q+1} \int_{0}^1 \partial_t^{q+1} \left(K^0(t) u(x, \eta t)\right) \psi_j\left(\frac{\eta t}{\varepsilon^2}\right) dt\,. 
\end{align}
Since $\eta < 1$, we can moreover bound the absolute value of
$\partial_t^{q+1}(K^0(t) u(x, \eta t))$ as
\begin{align}\label{eq:prod_leibniz}
\left| \partial_t^{q+1} (K^0(t) u(x, \eta t)) \right|
&\le C  \max_{\rho \le q_t+1} \sup_{t \in [0, \eta]} |\partial_t^\rho u(x, t)|
\,.
\end{align}
Using the definition of $\Psi$, \cref{eq:psi}, the equality
\cref{eq:cq_ibp} and the fact that the averaging kernel $K^0(t)$ is zero for
$t < 0$ since $K^0(t) \in \mathbb{K}_0^{q_t, p_t}$, we thus find that
\begin{align}\label{eq:k_u_psi}
  \int_{0}^\eta K^0_\eta(t)  u(x, t) \Psi\left(\frac{x}{\varepsilon}, \frac{t}{\varepsilon^2}\right) dt
&= \sum_{j=1}^\infty \boldsymbol \chi_j \phi_j\left(\frac{x}{\varepsilon}\right) \int_{0}^\eta K^0_\eta(t) u(x, t) \psi_j\left(\frac{t}{\varepsilon^2}\right) dt \\
&=  \int_0^1  \partial_t^{q+1} \left(K^0(t) u(x, \eta t)\right) g^\varepsilon(x, \eta t) dt\,, \nonumber  
\end{align}
where
\begin{align*}
  g(x, t) := \sum_{j=1}^\infty c_j^{q+1} \boldsymbol \chi_j \phi_j(x) \psi_j(t) \qquad\text{and} \qquad g^\varepsilon(x, t) = g(x/\varepsilon, t/\varepsilon^2)\,.
\end{align*}
Furthermore, we let for shortness of notation,
$k_\mu^\varepsilon(x) := \left(\partial_x^{\beta_1}
f(x/\varepsilon)\right) \left(\partial_x^{\beta_2} K_\mu(x)\right)$.
It then follows by \cref{eq:prod_leibniz} and \cref{eq:k_u_psi} that
\begin{align*}
    \left| \int_{0}^\eta   K^0_\eta(t)   \int_{\Omega_\mu}\right. & \left. \left(\partial_x^{\beta_1} f(x/\varepsilon)\right) \left(\partial_x^{\beta_2} K_\mu(x)\right)  u(x, t) \Psi\left({x}/{\varepsilon}, {t}/{\varepsilon^2}\right) dx dt \right| \\
&=
\left| \int_{\Omega_\mu} k_\mu^\varepsilon(x) \int_{0}^\eta K^0_\eta(t)  u(x, t) \Psi\left(x/\varepsilon, t/\varepsilon^2\right) dt dx \right| \\
&\le \int_0^1 \int_{\Omega_\mu} \left| \partial_t^{q+1} \left(K^0(t) u(x, \eta t)\right) \right|
   \left| k_\mu^\varepsilon(x) g^\varepsilon(x, \eta t)
\right| dx dt \\
&\le C  \max_{\rho \le q_t+1} \sup_{t \in [0, \eta]} \|\partial_t^\rho u(\cdot, t)\|_{L^\infty} \int_{0}^1
\int_{\Omega_\mu}\left|  k_\mu^\varepsilon(x) g^\varepsilon(x, \eta t)
 \right| dx dt
\,.
\end{align*}
Rescaling of the spatial integral and application of the Cauchy-Schwarz inequality yields
\begin{align*}
  &\int_{\Omega_\mu} | k^\varepsilon_\mu (x) g^\varepsilon(x, \eta t))| dx
    = \int_{[-1, 1]^d} \left| \mu^{-|\beta_1| - |\beta_2|} \partial_x^{\beta_1} f(\mu x/\varepsilon) \partial_x^{\beta_2} K(x) g^\varepsilon(\mu x, \eta t) \right| dx
\\ &\quad\le \left(\int_{[-1, 1]^d} |\mu^{-|\beta_1| - |\beta_2|} \partial_x^{\beta_1} f(\mu x/\varepsilon) \partial_x^{\beta_2} K(x)|^2 dx\right)^{1/2} \left( \int_{[-1, 1]^d} |g^\varepsilon(\mu x, \eta t)|^2 dx\right)^{1/2}.
\end{align*}
Note that $g(x, t)$ is 1-periodic in space and $\mu > \varepsilon$, which implies that the
latter integral can be bounded by the corresponding
$\|\cdot\|_{L^2(Y)}$-norm using \Cref{lemma:l2_per} below,
\begin{align*}
\int_{[-1, 1]^d} |g^\varepsilon(\mu x, \eta t)|^2 dx =
  \int_{[-1, 1]^d} \left|g\left({\mu x}/{\varepsilon}, {\eta t}/{\varepsilon^2}\right)\right|^2 dx
\le  C \left\|g\left(\cdot , {\eta t}/{\varepsilon^2}\right)\right\|_{L^2(Y)}^2\,.
\end{align*}
Since the absolute values of the eigenvalues $\omega_j$ are
increasing with $j$, all $c_j$, $j \ge 1$ can be bounded by
$c_1$, the constant involving $\omega_1$ the smallest non-zero eigenvalue
\begin{align}\label{eq:c_j_bound}
  |c_j| = \left|\frac{\varepsilon^2}{(\alpha - i)\omega_j \eta}\right|
= \frac{1}{\omega_j} \frac{\varepsilon^2}{\eta \sqrt{1 + \alpha^2}} \le
 \frac{1}{\omega_1} \frac{\varepsilon^2}{\eta \sqrt{1 + \alpha^2}} = |c_1|\,, \qquad j = 1, 2, ...\,.
\end{align}
One can therefore show using the orthogonality of the basis functions
$\phi_j$ and the boundedness of $\psi_j$, that
\begin{align*}
\left\|g\left(\cdot ,  {\eta t}/{\varepsilon^2}\right)\right\|_{L^2(Y)}^2
&= \sum_{j=1}^\infty \left|c_j^{q_t+1} \boldsymbol \chi_j \psi_j\left({\eta t}/{\varepsilon^2}\right)  \right|^2 
\le |c_1|^{2(q_t+1)} \sum_{j=1}^\infty |\boldsymbol \chi_j|^2 \\ &= |c_1|^{2(q_t+1)} \|\boldsymbol \chi\|^2_{L^2(Y)}\,,
\end{align*}
for any time $t \ge 0$.
It furthermore holds that
\begin{align*}
  \int_{[-1, 1]^d} |\mu^{-|\beta_1| - |\beta_2|} \partial_x^{\beta_1} f(\mu x/\varepsilon) \partial_x^{\beta_2} K(x)|^2 dx
  \le  \frac{1}{\mu^{2|\beta_2|} \varepsilon^{2|\beta_1|}} \sup_{y \in Y} |\partial_y^{\beta_1} f(y)|^2 \|K\|^2_{H^{|\beta_2|}}.
\end{align*}
Therefore, we obtain
\begin{align*}
   \int_{0}^1 \int_{\Omega_\mu} &| k^\varepsilon_\mu (x) g^\varepsilon(x, \eta t))| dx dt  \le
   \|\boldsymbol \chi\|_{L^2} \|K\|_{H^{|\beta_2|}} \|K^0\|_{W^{1, q_t +1}} \frac{|c_1|^{(q_t+1)}}{\mu^{|\beta_2|} \varepsilon^{|\beta_1|}} \sup_{y \in Y} |\partial_y^{\beta_1} f(y)|
\end{align*}
and the result in the lemma follows.
\end{proof}

In order to derive the above result, we used the following lemma from \cite{henrik1},
which is a useful tool when working with periodic functions.
For the convenience of the reader, a
short proof is given here.

\begin{lemma}\label{lemma:l2_per}
  Assume  $g \in L^2(Y)$ is 1-periodic.
  Let $\theta \ge 1$ and $a < b$ be given real constants, then
  \[\int_{[a, b]^d} |g(\theta x)|^2 dx \le  C\|g\|_{L^2(Y)}^2\,,\]
  where $C$ only depends on $a, b$ and the dimension $d$ but is independent of $\theta$.
\end{lemma}
\begin{proof} Let
$K = \lfloor (b-a)\theta \rfloor $,
the number of full periods of $g(\theta x)$ in the interval $[a, b]$ (in one coordinate direction).
Consider first a rescaled integral in the $j$th coordinate direction. As
$a\theta + (K+1) > b\theta$ and $|g|^2 > 0$, it holds that
\begin{align*}
 \int_{a\theta}^{b\theta} |g(x)|^2 dx_j
  \le 
  \sum_{k=0}^K \int_{a\theta + k}^{a\theta + (k+1)} |g(x)|^2 dx_j
&= (K+1) \int_0^1 |g(x)|^2 dx_j.
\end{align*}
Therefore, we find that
\begin{align*}
  \int_{[a, b]^d} |g(\theta x)|^2 dx
  &= 
    \frac{1}{\theta^d} \int_{a\theta}^{b\theta} \cdots \int_{a\theta}^{b\theta} |g(x)|^2 dx_1 \cdots dx_d \\
&\le \frac{K+1}{\theta^d} \int_{a\theta}^{b\theta} \cdots \int_0^1 |g(x)|^2 dx_1 \cdots dx_d \le \left(\frac{K + 1}{\theta}\right)^d \|g\|_{L^2(Y)}^2.
\end{align*}
Since $\theta \ge 1$,
\[\frac{K + 1}{\theta} \le {b-a} + \frac{1}{\theta} \le {b-a} + 1\,,\]
which entails that the constant multiplying $\|g\|_{L^2(Y)}^2$ is independent of $\theta$.
\end{proof}

\section{HMM approximation errors}\label{sec:avg_ll}

In this section, we prove bounds for the averaging error in each of
the three models (M1), (M2) and (M3) described in \Cref{sec:hmm}.
These error bounds depend on the parameters of the kernels used for
averaging in time and space, $p_x, q_x, p_t$ and $q_t$ as well as
the sizes of the averaging domains. Given a sufficiently regular
solution $\m_0$, which we assume in (A5), choosing high values for
the kernel parameters makes it possible to reduce the averaging error to
$\mathcal{O}(\varepsilon)$ as stated in the following theorem.

\begin{theorem}\label{thm:hmm_error}
  Assume (A1)-(A5) hold and let $\A^H$ be the homogenized
  coefficient matrix corresponding to $a^\varepsilon(x)$, given in
  \cref{eq:AH}.  Consider $\bar \K_{\eta, \mu}$ with averaging
  kernels $K \in \mathbb{K}^{p_x, q_x}$ and
  $K^0 \in \mathbb{K}^{p_t, q_t}_0$ and let $\varepsilon < \mu < 1$
  and $\varepsilon < \eta < \min(1, T^\varepsilon) $, where
  $T^\varepsilon$ is given by \cref{eq:Teps}.  Then
\begin{align*}
  &\left|\bar \K_{\mu, \eta} (a^\varepsilon \bnabla \m^\varepsilon) - \bnabla \m_0(0,0) \A^H \right|
 \\ &\hspace{4cm}\le C
 \left( \varepsilon + \left(\frac{\varepsilon}{\mu}\right)^{q_x + 2} + \mu^{p_x + 1} + \eta^{p_t + 1} + \left(\frac{\varepsilon^2}{\eta}\right)^{q_t+1}
\right)\,, \\
&\left|\bar \K_{\mu, \eta} (\bnabla \cdot (a^\varepsilon \bnabla \m^\varepsilon)) - \bnabla \cdot (\bnabla \m_0(0,0) \A^H) \right| \\ &\hspace{4cm}\le  C
 \left(\varepsilon + \left(\frac{\varepsilon}{\mu}\right)^{q_x + 2} + \mu^{p_x + 1} + \eta^{p_t + 1} + \frac{1}{{\mu}}  \left(\frac{\varepsilon^2}{ \eta}\right)^{q_t+1}
\right)\,,
\end{align*}
and
\begin{align*}
&\left|\bar \K_{\mu, \eta} (\m^\varepsilon \times \bnabla \cdot (a^\varepsilon \bnabla \m^\varepsilon))- \m_0(0,0) \times \bnabla \cdot (\bnabla \m_0(0,0) \A^H)\right| \\
 &\hspace{4cm} \le C
 \left(\varepsilon + \left(\frac{\varepsilon}{\mu}\right)^{q_x + 2} + \mu^{p_x + 1} + \eta^{p_t + 1} + \frac{1}{{\mu}}  \left(\frac{\varepsilon^2}{\eta}\right)^{q_t+1}
\right)\,.
\end{align*}
In all three cases, the constant $C$ is independent of $\varepsilon, \mu$ and $\eta$ but might depend on $K$, $K^0$ and $T$.
\end{theorem}

Given that the initial data to the micro problem is chosen such that
$\partial_x^\beta \m(0, 0)$, $|\beta| \le 2$ agree with the
corresponding derivatives of the macro solution $\M$ at the point in
time and space that one averages around, \Cref{thm:hmm_error}
provides estimates for $E_i$, $i = 1, 2, 3$ as given in \Cref{sec:hmm}.

To prove the estimates in \Cref{thm:hmm_error}, we consider an
approximation
$\m^\mathrm{app} := \m_0(x,t) + \varepsilon \m_1(x, x/\varepsilon,
t, t/\varepsilon^2)$ to $\m^\varepsilon$.
We then proceed in a similar way
for all three models. We first show that averaging of
$\m^\mathrm{app}$ results in approximations to the quantities
required to complete the models up to a certain error. The
contribution of $\m^\varepsilon - \m^\mathrm{app}$ only gives a
remainder term resulting in an additional error.  More precisely, it
holds for the approximation error in the first model, (M1), that
\begin{align}\label{eq:app1_split}
  \begin{split}
    |\bar \K_{\mu, \eta} (a^\varepsilon \bnabla \m^\varepsilon) - \bnabla \m_0(0,0) \A^H|
    &\le  |\bar \K_{\mu, \eta} (a^\varepsilon \bnabla \m^\mathrm{app}) - \bnabla \m_0(0,0) \A^H| \\ &\quad+ |\bar \K_{\mu, \eta} (a^\varepsilon \bnabla (\m^\varepsilon - \m^\mathrm{app})| \\ &=: e_{M1} + r_{M1}.
  \end{split}
\end{align}
Similarly, we have for (M2) that
\begin{align}\label{eq:app2_split}
  |\bar \K_{\mu, \eta} & (\L \m^\varepsilon) - \bnabla \cdot (\bnabla \m_0(0,0) \A^H)| \nonumber
  \\ &\le  |\bar \K_{\mu, \eta} (\L \m^\mathrm{app}) - \bnabla \cdot(\bnabla \m_0(0,0) \A^H)|
 + |\bar \K_{\mu, \eta} (\L (\m^\varepsilon - \m^\mathrm{app}))| \\&=: e_{M2} + r_{M2} \,, \nonumber
\end{align}
and in case of (M3)
\begin{align}\label{eq:app3_split}
  \begin{split}
     |\bar \K_{\mu, \eta} & \left(\m^\varepsilon \times \bnabla \cdot (a^\varepsilon \bnabla \m^\varepsilon)\right)
- \m_0(0,0) \times \bnabla \cdot (\bnabla \m_0(0,0) \A^H)| \\ &\le
|\m^\mathrm{app} \times \L \m^\mathrm{app} - \m_0(0,0) \times \bnabla \cdot (\bnabla \m_0(0,0) \A^H)| \\& +
  |\m^\mathrm{app} \times \L (\m^\varepsilon - \m^\mathrm{app}) +
  (\m^\varepsilon - \m^\mathrm{app}) \times \L \m^\mathrm{app}| =: e_{M3} + r_{M3}\,.
  \end{split}
\end{align}
Each of the approximation errors $e_{Mi}$, $i = 1, 2, 3$, is then
bounded using \Cref{lemma:avg_space} -- \Cref{lemma:avg_v_base} as
shown in the following sections. Finally, estimates for the norms of
the remainder terms $r_{Mi}$ are given in \cref{sec:remainder},
which completes the proof of \Cref{thm:hmm_error}.

For the derivations, we define the $y$-periodic functions
\begin{equation}
  \label{eq:def_g}
  \g(y) := a(y) (\I + \bnabla_y\boldsymbol \chi(y))\,, \quad \text{and} \quad \h(y) = a(y)\boldsymbol \chi(y).
\end{equation}
Since by assumption $(A1)$, $a \in C^\infty(Y)$, the same holds for
$\boldsymbol \chi$ and thus also $\g, \h \in C^\infty(Y)$.  Note
that $\g$ is a matrix-valued function,
$\g: \Real \to \Real^{d \times d}$, and we denote its elements by
$\g_{ij}$.
By the definition of the homogenized matrix $\A^H$ in \cref{eq:AH}, it holds that
\begin{equation}
  \label{eq:avg_g}
  \A^H_{ij} = \int_Y \g_{ij}(y) dy \,.
\end{equation}
The average of $\h$ is in general non-zero.
In the following, we use the notation
$\g^\varepsilon := \g(x/\varepsilon),
\h^\varepsilon:=\h(x/\varepsilon)$ and
$\boldsymbol \chi^\varepsilon := \boldsymbol \chi(x/\varepsilon)$.
Moreover, we define
\[\u(x, y, t) := \m_0(x,t) + \varepsilon \bnabla_x \m_0(x,t)\boldsymbol \chi(y),\]
and let $\u^\varepsilon := \u(x, x/\varepsilon, t)$.  Together with
\cref{eq:m1}, this implies that
\begin{equation}\label{eq:u_v_split}
\m^\mathrm{app} = \m_0(x,t) + \varepsilon \m_1(x, x/\varepsilon, t, t/\varepsilon^2)  = \u^\varepsilon + \varepsilon \v^\varepsilon,
\end{equation}
where $\v^\varepsilon = \v(x, x/\varepsilon, t, t/\varepsilon^2)$,
for $\v$ as given by \Cref{lemma:v}, and
\begin{equation}
  \label{eq:agrad_mapp}
     a^\varepsilon \bnabla \u(x, x/\varepsilon, t) =
   \bnabla_x \m_0(x,t) \g(x/\varepsilon) + \varepsilon \bnabla_x (\bnabla_x \m_0(x,t) \h(x/\varepsilon)).
\end{equation}
%
Furthermore, it follows from the definition of $\boldsymbol \chi$ via the cell
problem, 
that the divergence of $\g$ is zero,
\begin{equation}\label{eq:div_g}
\bnabla_y \cdot \g(y) =   \bnabla_y \cdot (a(y) (\I + \bnabla_y\boldsymbol \chi(y)))
=\bnabla_y a(y) + \bnabla_y \cdot (a(y) \bnabla_y\boldsymbol \chi(y)) = \boldsymbol 0\,.
\end{equation}
Hence, we have
\begin{align}\label{eq:L_mapp}
  \L \m^\mathrm{app} &=
  \bnabla \cdot \left(a(x/\varepsilon) \bnabla \left(\u(x, x/\varepsilon, t)
  + \varepsilon \v(x, x/\varepsilon, t, t/\varepsilon^2)\right)\right) \nonumber \\&=
  \bnabla_x\cdot(\bnabla_x \m_0 \g^\varepsilon) + \bnabla_y \cdot ( \bnabla_x (\bnabla_x \m_0 \h^\varepsilon)) + \varepsilon \bnabla_x \cdot (\bnabla_x (\bnabla_x \m_0 \h^\varepsilon)) + \varepsilon \L \v^\varepsilon.
\end{align}

Finally, for matters of brevity, we define a short-hand notation
for the error terms in \Cref{lemma:avg_space} and \Cref{lemma:avg1},
\begin{align}
  e_{44} (\varepsilon, \mu) :&= \left(\frac{\varepsilon}{\mu}\right)^{q_x +2} + \mu^{p_x + 1},\\
  e_{46}(\varepsilon, \mu, \eta) :&=  \left(\frac{\varepsilon}{\mu}\right)^{q_x +2} + \mu^{p_x + 1}  + \eta^{p_t + 1} .
\end{align}

 \subsection{Approximating  $\bnabla \m_0 \A^H$}

 We now apply the lemmas from the previous sections to obtain an
 estimate for the error $e_{M1}$ as given by  \cref{eq:app1_split}.

 \begin{lemma}\label{lemma:M1}
   Under the assumptions in \Cref{thm:hmm_error}, it holds that
\begin{align*}
|\bar \K_{\mu, \eta} (a^\varepsilon \bnabla \m^\mathrm{app}) &- \bnabla \m_0(0,0) \A^H| \\
  &\le  C \left(\left(\frac{\varepsilon}{\mu}\right)^{q_x +2} + \mu^{p_x + 1} + \eta^{p_t + 1}
+ \left(\frac{\varepsilon^2}{\eta }\right)^{q_t+1} + \varepsilon
 \right)\,,
\end{align*}
where $C$ is independent of $\varepsilon, \mu$ and $\eta$.
\end{lemma}

\begin{proof}
  We start by splitting $e_{M1}$ according to \cref{eq:u_v_split},
  into a part without fast oscillations in time and a part
  containing $\v$,
  \begin{equation}\label{eq:e1_bound}
    e_{M1} \le \left|\bar \K_{\mu, \eta} (a^\varepsilon \bnabla \u^\varepsilon) - \bnabla \m_0(0,0) \A^H \right| + \left|\bar \K_{\mu, \eta} (\varepsilon a^\varepsilon \bnabla \v^\varepsilon)\right| = I_1 + I_2.
  \end{equation}
  In the following we use the notation $\g_\beta$ and $\A^H_\beta$
  to refer to the element in the corresponding matrix that
  multiplies $\partial^\beta \m_0$ according to \cref{eq:agrad_mapp}
  and similarly for $\h$.  It then follows by \cref{eq:agrad_mapp}
  that there are constants $c_\beta, d_\beta$ such that
  \[a^\varepsilon \bnabla \u^\varepsilon = \sum_{|\beta| = 1} c_\beta \partial^\beta_x \m_0(x,t) \g_\beta(x/\varepsilon) +\varepsilon \sum_{|\beta| = 2} d_\beta \partial^\beta_x \m_0(x,t) \h_\beta(x/\varepsilon). \]
  Since by assumption (A5), we have
  $\m_0(x,t) \in C^\infty(0,T; H^\infty(\Omega))$, an application of
  \Cref{lemma:avg1} together with \cref{eq:avg_g} yields that given a
  multi-index $\beta$, 
  \begin{equation}
  \label{eq:I1}
  |\bar \K_{\mu, \eta} (\partial^\beta_x \m_0 \g^\varepsilon_{\beta}) - \partial^\beta_x \m_0(0, 0) \A^H_{\beta}
| \le C e_{46}(\varepsilon, \mu, \eta).
\end{equation}
As furthermore by the boundedness of $\bar \K_{\eta, \mu}$,
\cref{eq:k_inf_bound} in \Cref{lemma:k_boundedness},
\begin{equation*}
  |\bar \K_{\mu, \eta} (\partial^\beta_x \m_0 \h_\beta^\varepsilon)| \le C,
\end{equation*}
\cref{eq:I1} implies that
\[I_1 \le C   \left(\varepsilon + e_{46}(\varepsilon, \mu, \eta)\right). \]
The remaining term in \cref{eq:e1_bound}, $I_2$, is rewritten using
integration by parts, which
together with the definition of $\v$ according to \Cref{lemma:v}
 yields
\begin{equation*}\label{eq:v_start}
  \begin{split}
I_2
&= \varepsilon \left|\Re\left(\int_{\Omega_\mu} \int_{-\eta}^\eta K^0_\eta(t)   \nabla \left(a\left({x}/{\varepsilon}\right) K_\mu(x)\right) \f_\v(x,  t) \Psi\left({x}/{\varepsilon}, t/{\varepsilon^2}\right) dt dx\right)\right|.
  \end{split}
\end{equation*}
The regularity assumption (A5) for $\m_0$ implies that
$\f_\v \in C^\infty(0, T; H^\infty(\Omega))$ and
\[\max_{r \le q_t +1} \sup_{t \in [0, \eta]} \|\partial_t^r \f_v(\cdot, t)\|_{L^\infty} \le C,\]
wherefore it follows by \Cref{lemma:avg_v_base} that
\begin{align*}
 |I_2|
  &\le C
    \left(\frac{\varepsilon}{\mu} \|a\|_{L^\infty} + \|\partial_y a\|_{L^\infty} \right)\left(\frac{\varepsilon^2}{ \eta}\right)^{q_t + 1}
    \le  C \left(\frac{\varepsilon^2}{\eta}\right)^{q_t + 1}.
\end{align*}
Hence, we obtain the estimate in the lemma.
\end{proof}

\subsection{Approximating  $\bnabla \cdot (\bnabla \m_0 \A^H)$}

To estimate the first part of the approximation error in the second
model, $e_{M2}$ as given  in \cref{eq:app2_split}
we 
proceed in a similar way as before. However, as we consider the
divergence of the gradient, more terms need to be estimated. This
results in the following lemma.

 \begin{lemma}\label{lemma:M2} 
   Under the assumptions in \Cref{thm:hmm_error} it holds that
   \begin{align*}
    |\bar \K_{\mu, \eta} (\bnabla \cdot (a^\varepsilon \bnabla \m^\mathrm{app})) &- \bnabla \cdot (\bnabla \m_0(0,0) \A^H)| \\&\le C
 \left( \left(\frac{\varepsilon}{\mu}\right)^{q_x + 2} + \mu^{p_x+1} + \eta^{p_t + 1} + \frac{1}{\mu}\left(\frac{\varepsilon^2}{\eta}\right)^{q_t + 1}
 + \varepsilon
\right),
   \end{align*}
where $C$ is independent of $\varepsilon, \mu$ and $\eta$.
\end{lemma}

\begin{proof}
To begin with, we again split the error under consideration into two parts,
\begin{align*}
 e_{M2} &\le \left|\bar \K_{\mu, \eta} (\L \u^\varepsilon)
- \bnabla \cdot(\bnabla \m_0(0,0) \A^H)\right|
+ |\bar \K_{\mu, \eta} (\varepsilon \L \v^\varepsilon)|
=: II_1 + II_2.
\end{align*}
Based on \cref{eq:L_mapp} one can deduce that
there are constant coefficients $c_{\beta}, d_\beta$ such that
\begin{equation}
  \label{eq:Lu}
  \begin{split}
    \L \u^\varepsilon &= \L (\m_0 + \varepsilon \bnabla \m_0 \boldsymbol \chi^\varepsilon)
    \\&= \bnabla_x\cdot(\bnabla_x \m_0 \g^\varepsilon) + \bnabla_y \cdot ( \bnabla_x (\bnabla_x \m_0 \h^\varepsilon)) + \varepsilon \bnabla_x \cdot (\bnabla_x (\bnabla_x \m_0 \h^\varepsilon)) \\
    &= \sum_{|\beta| = 2} c_{\beta} [\g_\beta + (\bnabla_y \h)_\beta](x/\varepsilon) \partial_x^\beta \m_0(x,t) + \varepsilon
    \sum_{|\beta| = 3} d_\beta \h_\beta(x/\varepsilon) \partial_x^\beta \m_0(x,t).
  \end{split}
\end{equation}
Similar to before, $\g_\beta$ here denotes the component of $\g$ that
multiplies $\partial_x^\beta \m$ and accordingly for the other
quantities.  Hence, it holds that
\begin{align*}
II_1 &\le  \sum_{|\beta| = 2} c_\beta \left|\bar \K_{\mu, \eta} \left(\g_\beta^\varepsilon \partial_x^\beta \m_0\right)
- \A^H_\beta \partial^\beta_x \m_0(0,0)\right|+ \sum_{|\beta| = 2} c_\beta \left|\bar \K_{\mu, \eta} \left((\bnabla_y \h^\varepsilon)_\beta \partial_x^\beta \m_0\right)\right|
\\&\qquad + \varepsilon \sum_{|\beta| = 3} d_\beta \left|\bar \K_{\mu, \eta} \left(\h_\beta^\varepsilon \partial^\beta \m_0\right)\right|.
\end{align*}
Note that since $\h$ is $y$-periodic, the average with respect to
$y$ of the second term on the right-hand side here is zero.  The
averages in the other two sums can be bounded using \cref{eq:I1} and
\cref{eq:k_inf_bound} in \Cref{lemma:k_boundedness} and it follows
in the same way as in the estimate of $e_{M1}$ that
\[II_1  \le _C \left(  \varepsilon + e_{46}(\varepsilon, \mu, \eta) \right). \]
Furthermore, using integration by parts and \Cref{lemma:v}, we can rewrite $II_2$ as
\begin{align*}
  II_2 &= \varepsilon \left|\int_{-\eta}^\eta K^0_\eta(t) \int_{\Omega_\mu} \left(L K_\mu(x)\right) \v^\varepsilon  dx dt\right| \\
      &= \varepsilon \left| \Re\left( \int_{-\eta}^\eta K^0_\eta(t) \int_{\Omega_\mu} \nabla \cdot \left(a(x/\varepsilon) \nabla K_\mu\right)  \f_\v(x, t) \Psi\left(\frac{x}{\varepsilon}, \frac{t}{\varepsilon^2}\right)  dx dt\right)\right|,
\end{align*}
Thus it follows by \Cref{lemma:avg_v_base} with $|\beta_1| \le 1$
and $|\beta_2| = 2 - |\beta_1|$ that
\begin{align*}
   II_2 &\le C
     \left(\frac{\varepsilon}{\mu^2} \|a\|_\infty + \frac{1}{\mu} \|\partial_y a\|_\infty
    \right) \left(\frac{\varepsilon^2}{ \eta }\right)^{q_t + 1}
 \le C \frac{1}{\mu} \left(\frac{\varepsilon^2}{\eta}\right)^{q_t + 1}\,,
\end{align*}
which results in the estimate in \Cref{lemma:M2}.
\end{proof}

\subsection{Approximating $\m_0 \times \bnabla \cdot (\bnabla \m_0 \A^H)$}

We now consider $e_{M3}$, the first contribution to the error bound
for the approximation error in the third model as given in
\cref{eq:app3_split}.  Since we are now considering a nonlinear
expression, the derivations are more involved
than for the previous two models. However, the resulting estimate is
very similar to the one in \Cref{lemma:M2}, as stated in the following.

 \begin{lemma}\label{lemma:M3}
   Under the assumptions in \Cref{thm:hmm_error} it holds that
   \begin{align*}
    |\bar \K_{\mu, \eta} (\m^\mathrm{app} \times \L \m^\mathrm{app})) &- \m_0(0,0) \times \bnabla \cdot (\bnabla \m_0(0,0) \A^H)| \\&\le C
 \left( \left(\frac{\varepsilon}{\mu}\right)^{q_x + 2} + \mu^{p_x+1} + \eta^{p_t + 1} + \frac{1}{\mu}\left(\frac{\varepsilon^2}{\eta}\right)^{q_t + 1}
 + \varepsilon
\right),
   \end{align*}
where the constant $C$ is independent of $\varepsilon, \mu$ and $\eta$.
\end{lemma}

\begin{proof}
  For the term under consideration here, \cref{eq:u_v_split} implies
  that the error $e_{M3}$ can be split into four parts,
\begin{equation}
    \label{eq:mxLm_split}
    \begin{split}
    e_{M3} \le &|\bar \K_{\mu, \eta} (\u^\varepsilon \times \L \u^\varepsilon) - \m_0(0,0) \times \bnabla \cdot (\bnabla \m_0(0,0) \A^H) |
 \\ & + \varepsilon |\bar \K_{\mu, \eta} (\u^\varepsilon \times \L \v^\varepsilon)| + \varepsilon |\bar \K_{\mu, \eta} (\v^\varepsilon \times \L \u^\varepsilon)|
 +\varepsilon^2 |\bar \K_{\mu, \eta} (\v^\varepsilon \times \L \v^\varepsilon)| \\
 =&: III_1 + \varepsilon III_2 + \varepsilon III_3 + \varepsilon^2 III_4.
    \end{split}
\end{equation}
Using the sums in \cref{eq:Lu} to express $\L \u^\varepsilon$, we find
\begin{align}\label{eq:uxLu_sum}
  \begin{split}
  \u^\varepsilon \times \L \u^\varepsilon  &= (\m_0 + \varepsilon \sum_{|\gamma| = 1} \boldsymbol \chi_\gamma^\varepsilon \partial^\gamma_x \m_0) \times \L \u^\varepsilon
\\ &= \sum_{\substack{|\gamma| \le 1}} \left(\sum_{|\beta| = 2} c_\beta
 [(\varepsilon \boldsymbol \chi_\gamma)^{|\gamma|} (\g + \bnabla_y \h)_\beta ]\left(x/\varepsilon\right)
 [\partial_x^\gamma \m_0 \times \partial_x^\beta \m_0] (x, t) \right. \\
&\qquad\qquad+\left. \varepsilon \sum_{|\beta| = 3} d_\beta
[(\varepsilon \boldsymbol \chi_\gamma)^{|\gamma|} \h_\beta ]\left(x/\varepsilon\right)
[\partial_x^\gamma \m_0 \times \partial_x^\beta \m_0] (x, t) \right) \\
&=: \sum_{|\beta| = 2} c_\beta
 (\g + \bnabla_y \h)_\beta \left(x/\varepsilon\right)
 [\m_0 \times \partial_x^\beta \m_0] (x, t) + \varepsilon \f_0^\varepsilon(x, t)
 + \varepsilon^2 \f_1^\varepsilon(x, t),
  \end{split}
\end{align}
for some functions $\f_0^\varepsilon, \f_1^\varepsilon \in L^\infty(0, \eta; L^\infty(\Omega))$.
Since $\m_0 \in C^\infty(0, T; H^\infty(\Omega))$, we also have
$\partial_x^\gamma \m_0 \times \partial_x^\beta \m_0 \in C^\infty(0,T; H^\infty(\Omega)),$
and obtain using \Cref{lemma:avg1} that
\begin{equation*}
  |\bar \K_{\mu, \eta} ( [\m_0 \times \partial^\beta_x \m_0](x, t) \g_{\beta}(x/\varepsilon)) -  [\m_0 \times \partial^\beta_x \m_0](0, 0) \A^H_{\beta}
| \le C e_{46}(\varepsilon, \mu, \eta).
\end{equation*}
Furthermore, as $\h$ is $y$-periodic, the average of
$\partial^\nu_y \h_\beta$ for $|\nu| = 1$ is zero, hence
\begin{align*}
    \left|\bar \K_{\mu, \eta}  ( [ \m_0 \,\times \partial^\beta_x \m_0](x, t) \partial_y^\nu \h_\beta](x/\varepsilon))\right|
    \le  C e_{46}(\varepsilon, \mu, \eta).
\end{align*}
The remaining terms involved in \cref{eq:uxLu_sum} can be bounded
using \cref{eq:k_inf_bound} in \Cref{lemma:k_boundedness},
\[|\bar \K_{\mu, \eta} \f_0^\varepsilon | \le C, \qquad |\bar \K_{\mu, \eta} \f_1^\varepsilon|\le C.\]
In total, we can thus estimate $III_1$ as
\begin{align}\label{eq:nonlinear_est1}
  III_1 &= \left|\bar \K_{\mu, \eta} (\u^\varepsilon \times \L \u^\varepsilon)
  - \m_0(0,0) \times \bnabla \cdot (\bnabla \m_0(0,0) \A^H)  \right| \\ &\le C
\left(\varepsilon + e_{46}(\varepsilon, \mu, \eta)\right) \nonumber
\,.
\end{align}

The next two terms in \cref{eq:mxLm_split} can be bounded using
\Cref{lemma:avg_v_base}. Consider first $\varepsilon III_2$. Using
integration by parts as in the previous sections to obtain $\v$
without any spatial derivatives yields
\begin{align*}
   \int_{\Omega_\mu}  K_\mu \u^\varepsilon \times \L \v^\varepsilon dx
  &= \int_{\Omega_\mu} \left[\bnabla  \cdot (a^\varepsilon \bnabla (K_\mu \m_0))
+ \varepsilon  \bnabla \cdot  (a^\varepsilon \bnabla (K_\mu \bnabla \m_0 \boldsymbol \chi^\varepsilon))\right] \times \v^\varepsilon dx .
\end{align*}
Let now $\beta = \beta_1 + \beta_2 + \beta_3$ a multi-index. Using
\Cref{lemma:v}, we can rewrite
\begin{align*}
  &\bnabla  \cdot (a^\varepsilon \bnabla (K_\mu \m_0)) \times \v^\varepsilon
  \\ &\hspace{1cm}= - \Re\left(\sum_{|\beta_1| \le 1, |\beta| = 2} c_\beta \partial_x^{\beta_1} a(x/\varepsilon) \partial_x^{\beta_2} K_\mu(x) [\partial_x^{\beta_3} \m_0 \times \f_v](x, t) \Psi(x/\varepsilon, t/\varepsilon^2)\right)
\end{align*}
for some coefficients $c_{\beta}$ that might also be zero.  It
therefore follows by \Cref{lemma:avg_v_base}, the regularity
assumption (A5) and the fact that $\varepsilon < \mu$, that
\begin{align*}
  \varepsilon \int_{0}^\eta K_\eta  &\int_{\Omega_\mu} \bnabla  \cdot (a^\varepsilon \bnabla (K_\mu \m_0 )) \times \v^\varepsilon dx dt \\
  &\le C
    \sum_{\substack{|\beta_1| \le 1, \\|\beta_2| \le 2 - |\beta_1|}} \frac{\varepsilon}{\mu^{|\beta_2|}\varepsilon^{|\beta_1|}} \|\partial_y^{\beta_1} a\|_{L^\infty}  \left(\frac{\varepsilon^2}{ \eta} \right)^{q_t + 1}
  \le C \frac{1}{\mu} \left(\frac{\varepsilon^2}{\eta}\right)^{q_t + 1}.
  \end{align*}
Similarly, it holds
with $\gamma = \gamma_1 + \gamma_2 + \gamma_3 + \gamma_4$ that
that
\begin{align*}
  &\bnabla  \cdot (a^\varepsilon \bnabla (K_\mu \bnabla \m_0 \boldsymbol \chi^\varepsilon)) \times \v^\varepsilon
  \\ &\hspace{.1cm}= -\Re\left(\sum_{\substack{|\gamma_1| \le 1 \\ |\gamma| \le 2}} \sum_{|\nu| = 1} c_\gamma [\partial_x^{\gamma_1} a \partial_x^{\gamma_4} \boldsymbol \chi_\gamma] (x/\varepsilon) \partial_x^{\gamma_2} K_\mu(x) [\partial_x^{\gamma_3} \partial_x^\nu \m_0 \times \f_v](x, t) \Psi(x/\varepsilon, t/\varepsilon^2)\right)
\end{align*}
Application of \Cref{lemma:avg_v_base}, with $|\beta_1| = 0$ and
$f(y) = \varepsilon^{-|\gamma_1| - |\gamma_4|}
[\partial_y^{\gamma_1} a \partial_y^{\gamma_4} \boldsymbol
\chi_\gamma](y)$,
then yields
\begin{align*}
  \varepsilon^2 &\int_{0}^\eta K_\eta \int_{\Omega_\mu} \bnabla  \cdot (a^\varepsilon \bnabla (K_\mu \m_0)) \times \v^\varepsilon dx dt \\
      &\le C \sum_{\substack{|\gamma_1| \le 1, \\ |\gamma_2| + |\gamma_4| \le 2 - |\gamma_1|}}
    \left(\frac{\varepsilon^2}{\varepsilon^{|\gamma_1| + |\gamma_4|} \mu^{|\gamma_2|}} \|\partial_y^{\gamma_1} a \partial_y^{\gamma_4} \boldsymbol \chi\|_{L^\infty}
  \right) \left(\frac{\varepsilon^2}{ \eta}\right)^{q_t + 1}
  \le C \left(\frac{\varepsilon^2}{\eta}\right)^{q_t + 1}\,.
  \end{align*}
By (A1) and since $\varepsilon/\mu < 1$, we thus have in total
\begin{align}\label{eq:nonlinear_est2}
  \varepsilon III_2 &= \left| \bar \K_{\mu, \eta} (\varepsilon \u^\varepsilon \times \L \v^\varepsilon) \right| \le
C \frac{1}{\mu} \left(\frac{\varepsilon^2}{ \eta }\right)^{q_t + 1}.
\end{align}

The next term in \cref{eq:mxLm_split}, $\varepsilon III_3$,
 can be treated similarly. Expressing
$\L \u^\varepsilon$ as given in \cref{eq:Lu} and using \Cref{lemma:v},
we get
\begin{align*}
  \L\u^\varepsilon \times \v^\varepsilon
      &=  - \Re \Big( \sum_{|\beta| = 2} c_\beta [\g_\beta + (\bnabla_y \h)_\beta](x/\varepsilon) [\partial_x^\beta \m_0 \times \f_v](x,t) \Psi \left({x}/{\varepsilon}, {t}/{\varepsilon^2}\right)
\\ &\hspace{1.3cm}+ \varepsilon \sum_{|\beta| = 3} d_\beta  \h_\beta(x/\varepsilon) \left[\partial_x^\beta \m_0 \times \f_v\right](x, t) \Psi\left({x}/{\varepsilon}, {t}/{\varepsilon^2}\right) \Big).
\end{align*}
Application of \Cref{lemma:avg_v_base} thus yields that
\begin{align}\label{eq:nonlinear_est3}
   \varepsilon |III_3| &=
            \left|\int_{0}^\eta \int_{\Omega_\mu} K^0_\eta K_\mu \right.(\L\u^\varepsilon \times \v^\varepsilon) dx dt\Bigg|
\\ &\le C
     \left(\|\g\|_{L^\infty} + \|\bnabla_y \h\|_{L^\infty} + \varepsilon \|\h\|_{L^\infty}\right) 
     \varepsilon
\left(\frac{\varepsilon^2}{ \eta }\right)^{q_t + 1}.
     \nonumber
\end{align}

Finally, consider the last term in \cref{eq:mxLm_split},
$\v^\varepsilon \times \L \v^\varepsilon$, which can be bounded using \Cref{lemma:avg_space}
and the fact that $\v^\varepsilon$ decays exponentially in time.  As we mostly
consider spatial averaging only, we in the following suppress the
time dependence of $\v$ for matters of brevity and write $\v(x, y)$
instead of $\v(x, y, t, \tau)$.  Furthermore, we use the notation
\begin{align}\label{eq:L_xy}
  \L_{ab} \v := \bnabla_a \cdot (a(y) \bnabla_b \v).
\end{align}
Then we can write
\begin{align*}
  \varepsilon^2 \v^\varepsilon \times \L \v^\varepsilon &= \varepsilon^2 \v^\varepsilon \times \left[\L_{xx}  + {\varepsilon^{-1}} (\L_{xy} + \L_{yx} ) + \varepsilon^{-2} \L_{yy}\right] \v^\varepsilon \\
  &=: \v^\varepsilon \times \L_{yy} \v^\varepsilon + \varepsilon \hat \f_0^\varepsilon + \varepsilon^2 \hat \f_1^\varepsilon
  \,,
\end{align*}
Consider first the average of $\v \times \L_{yy} \v$. Using integration by
parts, we find that for any coordinate direction $\ell$,
\begin{align*}
  \int_0^1 \v \times \partial_{y_\ell} (a \partial_{y_\ell} \v) dy_\ell =
 - \int_0^1 a \partial_{y_\ell} \v \times  \partial_{y_\ell} \v dy_\ell = 0\,,
\end{align*}
which implies that the average of $\v \times \L_{yy} \v$ with
respect to $y$ is zero.
Moreover, we obtain using \cref{eq:sob} and \cref{eq:v_norm_decay} that
\begin{align*}
  \sup_{y \in Y} \|\v(\cdot, y) \times \L_{yy} \v(\cdot, y) \|_{W^{p_x+1, \infty}}   
  &\le C  \|\v\|_{H^{p_x+3, 2}}  \|\v\|_{H^{p_x+3, 4}}  \le C e^{- 2\gamma t/\varepsilon^2}.
\end{align*}
%
As $\v$ is periodic in $y$, spatial averaging according to
\Cref{lemma:avg_space} thus  yields
\begin{align*}
  \left|\int_{\Omega_\mu} K_\mu (\v^\varepsilon \times \L_{yy} \v^\varepsilon) dx \right|
  \le  C  e_{44}(\varepsilon, \mu) e^{-2 \gamma t/\varepsilon^2}.
\end{align*}
The averages of $\hat \f_0^\varepsilon$ and $\hat \f_1^\varepsilon$ can be bounded using
\cref{eq:k_inf_bound} in \Cref{lemma:k_boundedness}.
Consequently, it holds for $\varepsilon^2 III_4$ that
\begin{align}\label{eq:nonlinear_est4}
  \varepsilon^2 III_4 &\le \left| \int_0^\eta \int_{\Omega_\mu} K^0_\eta K_\mu (\v^\varepsilon \times \L_{yy} \v^\varepsilon) dx dt \right|
                        + \varepsilon \left|\bar \K_{\mu, \eta} \hat \f_0^\varepsilon\right|
                        + \varepsilon^2 \left|\bar \K_{\mu, \eta} \hat \f_1^\varepsilon\right|
  \\& \le C \left(\varepsilon+ \varepsilon^{2} +  e_{44}(\varepsilon, \mu)  \int_0^\eta \left|K^0_\eta\right|  e^{-2 \gamma t/\varepsilon^2}  dt \right) \le C \left(\varepsilon + e_{44}(\varepsilon, \mu)\right).
      \nonumber
\end{align}
Based on \cref{eq:mxLm_split} together with
\cref{eq:nonlinear_est1}, \cref{eq:nonlinear_est2},
\cref{eq:nonlinear_est3} and \cref{eq:nonlinear_est4}, we thus
obtain the estimate in \Cref{lemma:M3}.
\end{proof}

\subsection{Remainder terms}\label{sec:remainder}

In this section, we aim to bound the remainder errors
$r_{M1}, r_{M2}$ and $r_{M3}$ that are introduced in the averaging
processes of the models $(M1)-(M3)$ when approximating
$\m^\varepsilon$ by $\m^\mathrm{app}$ as shown in
\cref{eq:app1_split} - \cref{eq:app3_split}.  To be able to
formulate common results for these remainder terms, we introduce
four different linear operators,
\begin{align*}
  \AA_1 \u :&= a^\varepsilon \bnabla \u, &\qquad \AA_{2} \u :&= \L \u, \\
  \AA_3 \u :&= \m^\mathrm{app} \times \L \u, &\qquad   \AA_3 \u :&= \u \times  \L \m^\mathrm{app},
\end{align*}
where $\u \in H^2(\Omega)$ and
$\m^\mathrm{app} := \m_0 + \varepsilon \m_1^\varepsilon$ as in the previous
section. In terms of these operators, the remainder terms are
\begin{align*}
  r_{M1} &= \left|\bar \K_{\mu, \eta}( \AA_1 (\m^\varepsilon - \m^\mathrm{app}))\right|, \\
  r_{M2} &= \left|\bar \K_{\mu, \eta}( \AA_2 (\m^\varepsilon - \m^\mathrm{app}))\right|, \\
  r_{M3} &= \left|\bar \K_{\mu, \eta}( \AA_3 (\m^\varepsilon - \m^\mathrm{app})
           + \AA_4 (\m^\varepsilon - \m^\mathrm{app})\right|.
\end{align*}
To bound the remainder errors, we first consider the approximation
$\tilde \m_J^\varepsilon$ as given by \cref{eq:asymptotic_expansion}
and let $\e_J := \m^\varepsilon - \tilde \m_J^\varepsilon$. Note
that by definition, $\m^\mathrm{app} = \tilde \m_1^\varepsilon$. It thus holds that
\begin{equation}\label{eq:m_e_connection}
  \m^\varepsilon - \m^\mathrm{app} = \e_J + \sum_{j=2}^J \varepsilon^j \m_j^\varepsilon\,,
\end{equation}
for some constant $J > 1$, where
$\m_j^\varepsilon(x,t) := \m_j(x, x/\varepsilon, t, t/\varepsilon^2)$. As
explained in \Cref{sec:hom}, bounds for both $H^q$-norms of $\e_J$
as well as for the norms of $\m_j^\varepsilon$ have been proved in \cite{paper1}
and can be used to obtain estimates here.
It is important to notice that 
the correctors
$\m_j(x, y, t, \tau)$, $j > 0$, are periodic in $y$. We can
therefore apply \Cref{lemma:avg_space} to averages of these terms
and their derivatives. The error $\e_J$, on the other hand, can in
general not be assumed to be periodic and therefore has to be
treated differently.

According to \cref{eq:k_l2_bound} in \Cref{lemma:k_boundedness}, it
holds for the averages of the linear operators $\AA_k$,
$k = 1, ..., 4$ applied to $\e_J$ that
\begin{align}\label{eq:Ke_start}
  \left| \bar \K_{\mu, \eta} (\AA_k \e_J) \right| \le \sup_{0 \le t \le \eta} \frac{1}{\mu^{d/2}} \|\AA_k
  \e_J(\cdot, t)\|_{L^2} \le
  \sup_{0 \le t \le \eta} \frac{1}{\varepsilon^{d/2}} \|\AA_k
  \e_J(\cdot, t)\|_{L^2}.
\end{align}
Using the fact that by \cref{eq:tilde_mJ_infty}, for $0 \le t \le T^\varepsilon = \varepsilon^\sigma T$
with $1 < \sigma \le 2$,
\[\|\m^\mathrm{app}\|_{W^{q, \infty}} = \|\tilde \m_1^\varepsilon\|_{W^{q, \infty}} \le C \varepsilon, \qquad q \ge 1\,,\]
and the bound \cref{eq:error_mJbest} for $\|\e_J\|_{H^q}$,
the $L^2$-norms on the right-hand
side in \cref{eq:Ke_start} can be bounded for
$0 \le t \le T^\varepsilon$ as follows,
\begin{align*}
  \|\AA_1 \e_J(\cdot, t)\|_{L^2} &\le C \|\e_J(\cdot, t)\|_{H^1} \le C \varepsilon^{\sigma_J+1}, \\
  \|\AA_2 \e_J(\cdot, t)\|_{L^2} &\le C \left(\varepsilon^{-1} \|\e_J(\cdot, t)\|_{H^1} +  \|\e_J(\cdot, t)\|_{H^2}\right) \le C \varepsilon^{\sigma_J}, \\
  \|\AA_3 \e_J(\cdot, t)\|_{L^2} &\le C \|\m^\mathrm{app}\|_{W^{1, \infty}} \left(\varepsilon^{-1} \|\e_J(\cdot, t)\|_{H^1} +  \|\e_J(\cdot, t)\|_{H^2}\right) \le C \varepsilon^{\sigma_J}, \\
  \|\AA_4 \e_J(\cdot, t)\|_{L^2} &\le C \|\e_J(\cdot, t)\|_{L^2}\left(\varepsilon^{-1} \|\m^\mathrm{app}\|_{W^{1, \infty}} + \|\m^\mathrm{app}\|_{W^{2, \infty}}\right) \le C \varepsilon^{\sigma_J+1},
\end{align*}
where $\sigma_J := (\sigma-1)(J-1)$. We now choose
$J \ge (1 + d/2)/(\sigma - 1) + 1$. Then $\sigma_J \ge 1 + d/2$ and
since $\eta \le T^\varepsilon$, it follows together with
\cref{eq:Ke_start} that
\begin{align}\label{eq:Ke}
  \left|\bar \K_{\mu, \eta} (\AA_k \e_J) \right| \le C \varepsilon, \qquad k = 1, ..., 4.
\end{align}
This provides a bound for the first part of the remainder errors. To
also bound the second part, we use the $L^\infty$-boundedness of
$\bar \K_{\mu, \eta}$ as well as the periodicity of the correctors $\m_j$.
To simplify these considerations, we split the operator $\AA_3$,
\begin{align*}
  \AA_3 \u = \m_0 \times \L \u + \varepsilon \m_1 \times \L \u =: \AA_{31} \u + \AA_{32} \u\,.
\end{align*}
Note that by \cref{eq:mj_norm_bound}, we have
for the considered $T^\varepsilon$ and any $p, q > 0$ 
\begin{align}\label{eq:mj_norm_short}
  \sup_{0 \le t \le T^\varepsilon} \|\m_j(\cdot, \cdot, t, t/\varepsilon^2)\|_{H^{p, q}} \le
  C \varepsilon^{\min(0, 2-j)}, \qquad j \ge 1.
\end{align}
Together with Lemma 5.1 in \cite{paper1}, this implies that
\begin{align*}
  \sup_{0 \le t \le T^\varepsilon} \|\m_j^\varepsilon(\cdot, t)\|_{W^{q, \infty}}
  &\le
  \sup_{0 \le t \le T^\varepsilon} C \varepsilon^{-q} \|\m_j(\cdot, \cdot,  t, t/\varepsilon^2)\|_{H^{q+2, q+2}}
  \le C \varepsilon^{\min(0, 2-j) - q}.
\end{align*}
It therefore follows by \cref{eq:k_inf_bound} in \Cref{lemma:k_boundedness}
that for $j \ge 2$ and $0 \le t \le T^\varepsilon$,
\begin{align*}
  \left| \bar \K_{\mu, \eta} \AA_1 \m_j^\varepsilon\right|
  &\le C \|\AA_1 \m_j^\varepsilon\|_{L^\infty} \le
    C \|\m_j^\varepsilon\|_{W^{1, \infty}} \le C \varepsilon^{1-j}, \\
  \left| \bar \K_{\mu, \eta} \AA_{32} \m_j^\varepsilon\right|
  &\le C \|\AA_{32} \m_j^\varepsilon\|_{L^\infty}
    \le
    C\varepsilon \|\m_1^\varepsilon\|_{L^{\infty}}
    \left(\varepsilon^{-1} \|\m_j^\varepsilon\|_{W^{1, \infty}} + \|\m_j^\varepsilon\|_{W^{2, \infty}}\right)
    \le C \varepsilon^{1-j}, \\
 \left| \bar \K_{\mu, \eta} \AA_4 \m_j^\varepsilon\right|
  &\le C \|\AA_4 \m_j^\varepsilon\|_{L^\infty} \\ &\le
    C \|\m_j^\varepsilon\|_{L^{\infty}}
    \left(\varepsilon^{-1} \|\m^\mathrm{app}\|_{W^{1, \infty}} + \|\m^\mathrm{app}\|_{W^{2, \infty}}\right)
    \le C \varepsilon^{1-j},
\end{align*}
To  bound the remaining terms,
$ |\bar \K_{\mu, \eta} \AA_2 \m_j|$ and
$ |\bar \K_{\mu, \eta} \AA_{31} \m_j|$, we furthermore have to
exploit the periodicity of the correctors. We proceed in a similar
way as for $III_4$ in the previous section, applying
\Cref{lemma:avg_space} and again using the notation given in
\cref{eq:L_xy}. Suppressing time dependence, let
\[\f_j(x, y) := \left(\L_{xx} + \varepsilon^{-1} (L_{xy} + L_{yx}) + \varepsilon^{-2} \L_{yy}\right)
  \m_j(x, y)\]
As the $Y$-average of $\L_{yy} \m_j$ is zero, we
then find using \cref{eq:mj_norm_short} that for $j \ge 2$,
\begin{align*}
  \left| \int_Y \f_j(0, y) dy \right|
  \le C \left( \|\m_j\|_{H^{4, 0}} + \varepsilon^{-1} \|\m_j\|_{H^{3, 1}}\right) \le
  C \varepsilon^{1-j}
\end{align*}
and
\begin{align*}
\nonumber  &\sup_{y \in Y} \|\f_j(\cdot, y)\|_{W^{p_x+1, \infty}} \\ &\hspace{1cm}\le
  C  \left( \|\m_j\|_{H^{p_x+5, 2}} + \varepsilon^{-1} \|\m_j\|_{H^{p_x+4, 3}}
  + \varepsilon^{-2}  \|\m_j\|_{H^{p_x+3, 4}} \right)
  \le C \varepsilon^{-j}.
\end{align*}
Hence, the spatial average of $\AA_2 \m_j$ is according to
\Cref{lemma:avg_space} bounded as follows,
\begin{align*}
  \int_{\Omega_\mu} K_\mu(x) \AA_2 \m_j(x, x/\varepsilon) dx
  &\le \left| \int_Y \f_j(0, y) dy \right| + C \sup_{y \in Y} \|\f_j\|_{W^{p_x + 1, \infty}}
    e_{44}(\varepsilon, \mu)
    \\ &\le C \varepsilon^{-j} \left( \varepsilon + e_{44}(\varepsilon, \mu) \right),
\end{align*}
and similarly,
\begin{align*}
  \int_{\Omega_\mu} K_\mu \AA_{31} \m_j dx
  &\le \left| \m_0(0) \times \int_Y \f_j(0, y) dy \right| + C \sup_{y \in Y} \|\m_0 \times \f_j\|_{W^{p_x + 1, \infty}}
    e_{44}(\varepsilon, \mu)
    \\ &\le C \varepsilon^{-j} \left( \varepsilon + e_{44}(\varepsilon, \mu) \right).
\end{align*}
We can therefore conclude that for $\ell \in \{2, 31\}$,
\begin{align*}
  \left|\bar \K_{ \mu, \eta} \AA_\ell \m_j\right| \le \int_0^\eta |K^0_\eta| \left|\int_{\Omega_\mu}
  K_\mu \AA_\ell \m_j dx \right| dt
  \le C \varepsilon^{-j} \left( \varepsilon + e_{44}(\varepsilon, \mu) \right).
\end{align*}
Overall, we finally obtain that when choosing $J \ge (1 + d/2)/(\sigma-1) + 1$,
\begin{align*}
  r_{M_1} &\le \left| \bar \K_{\mu, \eta} \AA_1 \e_J \right| + \sum_{j=2}^J \varepsilon^j \left| \bar \K_{\mu, \eta} \AA_1 \m_j \right|
  \le C \varepsilon, \\
    r_{M_2} &\le \left| \bar \K_{\mu, \eta} \AA_2 \e_J \right| + \sum_{j=2}^J \varepsilon^j \left| \bar \K_{\mu, \eta} \AA_2 \m_j \right|
              \le C \varepsilon +  e_{44}(\varepsilon, \mu),\\
  r_{M_3} &\le \left| \bar \K_{\mu, \eta} (\AA_3 + \AA_4) \e_J \right|  + \sum_{j=2}^J \varepsilon^j \left| \bar \K_{\mu, \eta} (\AA_{31} + \AA_{32} + \AA_{4}) \m_j \right|
                \le C \varepsilon +  e_{44}(\varepsilon, \mu).
\end{align*}
This completes the proof of \Cref{thm:hmm_error}.

\section{Numerical results}\label{sec:num}
In this section, we present numerical examples showing the
convergence of the averaged flux, field and torque as specified in
the models (M1) - (M3) to the corresponding homogenized quantities
in both one and two space dimensions. We provide evidence for the
estimates given in \Cref{thm:hmm_error} and show which of the terms
appearing there seem to be dominating in practice.

\subsection{One-dimensional examples}

In one space dimension, we consider the periodic material coefficient
\begin{align}\label{eq:a_eps}
  a^\varepsilon(x) = a(x/\varepsilon), \qquad \text{where} \qquad a(y) =  1 + 0.5 \sin(2 \pi y) + 0.5 \sin(4 \pi y)\,.
\end{align}
The corresponding homogenized coefficient, given by
\begin{align*}
  a^H = \left(\int_0^1 (a(y))^{-1} dy \right)^{-1}\,,
\end{align*}
is computed numerically.
\begin{figure}[h!]
  \centering
  \includegraphics[width=.45\textwidth]{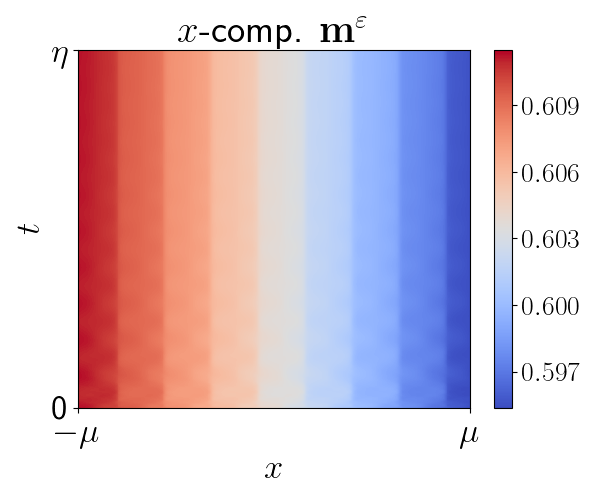}~
  \includegraphics[width=.45\textwidth]{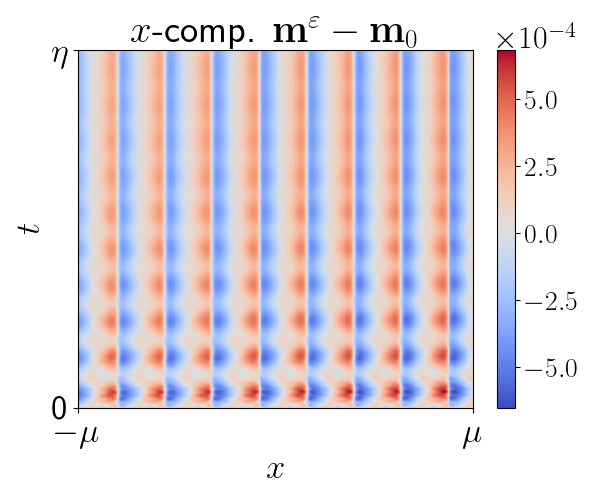}
  \caption{Left: Development of $x$-component of solution
    $\m^\varepsilon$ to \cref{eq:micro_prob} with $a^\varepsilon$
    given by \cref{eq:a_eps} in space and time when
    $\varepsilon = 1/140$ and $\alpha = 0.05$. Spatial domain
    $[-\mu, \mu]$ where $\mu = 0.03$, time interval $t \in [0, \eta]$ with $\eta = 1.5\cdot10^{-4}$.
    Right: difference between that solution $\m^\varepsilon$ and
    corresponding homogenized solution $\m_0$. }
  \label{fig:1d_sol}
\end{figure}
To create a better understanding of the
problem, \cref{fig:1d_sol} shows the $x$-component of the
solution $\m^\varepsilon$ to an example problem in time and space as
well as the difference between $\m^\varepsilon$ and the solution to
the corresponding homogenized equation, $\m_0$.  One can observe
that $\m^\varepsilon$ oscillates in both time and space initially,
but as $t$ increases the temporal oscillations are damped away and
only the spatial ones remain. The oscillations have significantly
smaller magnitude than the solution and appear to have zero average.

We then consider the approximation errors for the three models (M1) - (M3)
and compare the observed behavior with the theoretical bound according to
\Cref{thm:hmm_error},
\begin{align}\label{eq:err_est_summary}
  E_{Mi}  \le C \left(\varepsilon + \left(\frac{\varepsilon}{\mu}\right)^{q_x + 2} + \mu^{p_x + 1} + \eta^{p_t + 1} + \frac{1}{\mu^{\delta_{i}}} \left(\frac{\varepsilon^2}{\eta}\right)^{q_t+1} \right), \quad i = 1, 2, 3,
\end{align}
where $\delta_i = 0, i = 1$ and $\delta_i = 1, i = 2, 3$.  In
\cref{fig:1d_comp}, the approximation errors for varying
$\varepsilon$ is shown.  In \cref{fig:1d_comp_001}, the damping
constant is set to $\alpha = 0.01$ and in \cref{fig:1d_comp_01} we
have $\alpha = 0.1$. For all three models, the errors initially
decrease rapidly with $\varepsilon$. In this regime, the error
appears to be dominated by the $\varepsilon^2/\eta$ term in
\cref{eq:err_est_summary}. For smaller $\varepsilon$, the errors are
proportional to $\varepsilon^2$. Note that this is smaller than the
convergence rate of $\varepsilon$ suggested by
\cref{eq:err_est_summary}. In (M1) the error is somewhat lower than for
(M2) and (M3). The latter two models result in very similar error
behavior.
\begin{figure}[h!]
  \centering
    \begin{subfigure}[b]{\textwidth}
      \centering
      \includegraphics[width=\textwidth]{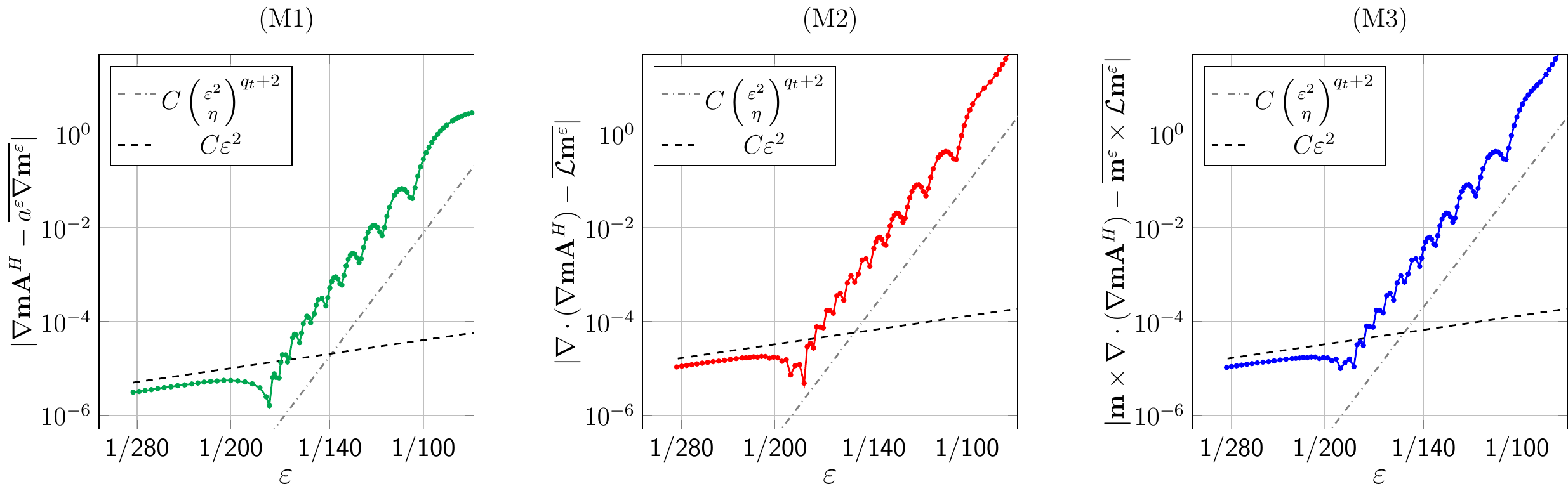}
      \caption{Low damping, $\alpha = 0.01$}
      \label{fig:1d_comp_001}
    \end{subfigure}

    \begin{subfigure}[b]{\textwidth}
      \centering
       \includegraphics[width=\textwidth]{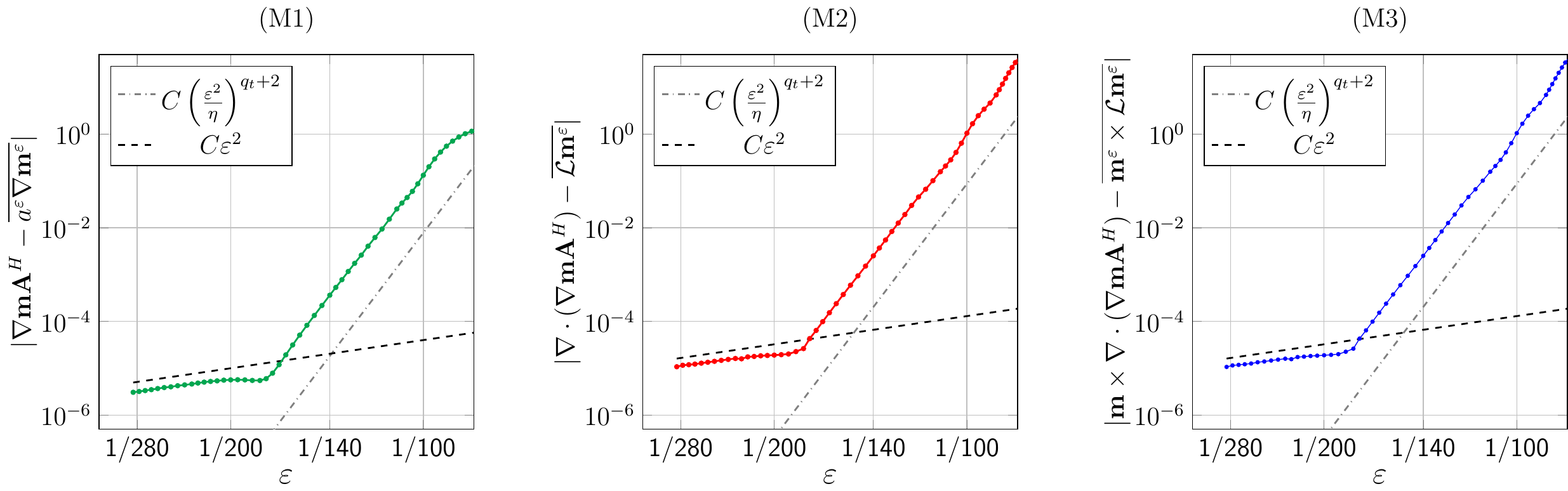}
      \caption{Higher damping, $\alpha = 0.1$}
      \label{fig:1d_comp_01}
    \end{subfigure}

    \caption{Approximation errors in (M1)-(M3) for varying
      $\varepsilon$ with kernel parameters
      $p_x = 5, q_x = 7, p_t = 5, q_t = 7$ and averaging domain
      sizes $\eta = 1.5\cdot10^{-4}$ and $\mu = 0.03$.}
  \label{fig:1d_comp}
\end{figure}

When comparing \cref{fig:1d_comp_001} and \cref{fig:1d_comp_01}, one
can moreover observe that choosing a lower damping parameter
$\alpha$ results in more oscillatory errors. However, the overall
error behavior is very similar for both $\alpha$-values.  This is a
property that holds for all the examples considered here.
We therefore choose to show plots for higher values of $\alpha$ in
several of the subsequent examples to reduce oscillations and make it
easier to distinguish the different curves.

Next, we consider the influence of the kernel parameters $q_x$ and
$q_t$ on the error decay. As shown in \Cref{fig:1d_qx_qt}, the
choice of $q_x$ does not influence the error behavior, while
different values of $q_t$ result in different slopes of the initial
error decay. This again shows that the error from time averaging initially
dominates in \Cref{fig:1d_comp} and \Cref{fig:1d_qx_qt}. In
particular, it is proportional to $(\varepsilon^2/\eta)^{q_t+2}$ until
it reaches the $\varepsilon^2$ threshold. This is slightly better
than $(\varepsilon^2/\eta)^{q_t+1}$ as given in \Cref{thm:hmm_error}.
\begin{figure}[h!]
  \centering
  \includegraphics[width=.85\textwidth]{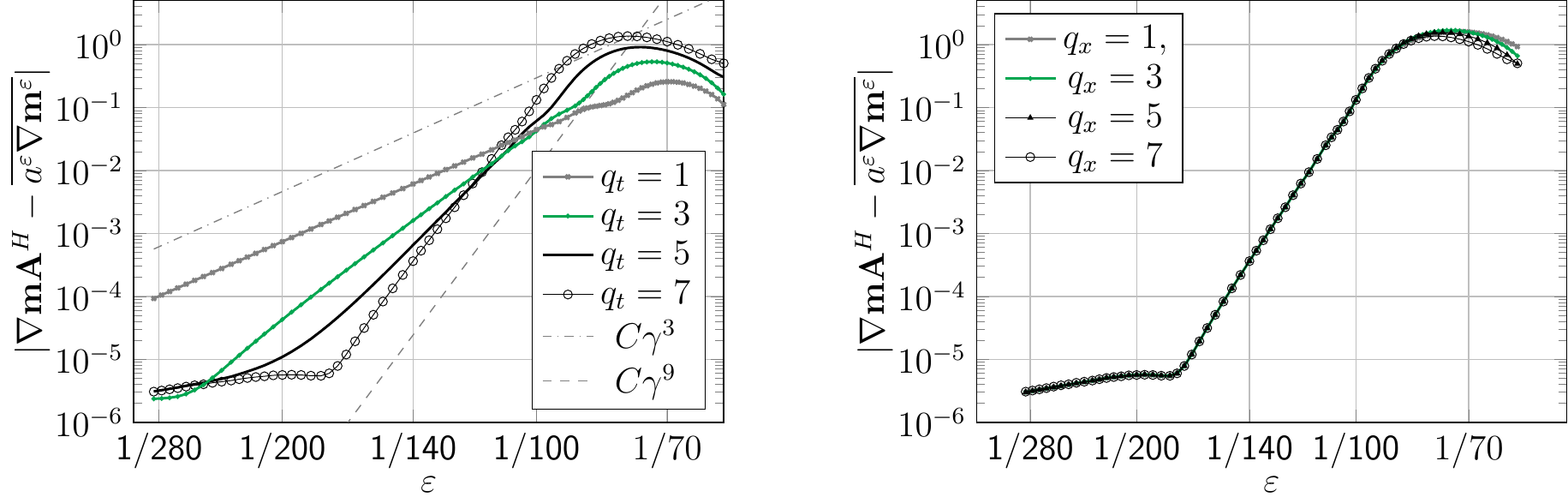}
  \caption{Approximation error for (M1) for varying $\varepsilon$
    with $\alpha = 0.1$ and kernel parameters
    $p_x = 5, q_x = 7, p_t = 5$ and $q_t = 7$, if not explicitly
    stated otherwise in the plot. Averaging domain sizes
    $\eta = 1.5\cdot10^{-4}$ and $\mu = 0.03$. For reference lines,
    $\gamma := \varepsilon^2/\eta$.}
  \label{fig:1d_qx_qt}
\end{figure}

We continue by examining the influence of the parameters $p_x$ and
$p_t$ and the contribution of the terms $C \mu^{p_x + 1}$ and
$C \eta^{p_t + 1}$ to the error. As shown in \cref{fig:1d_px_pt}, we
find that low choices of both $p_x$ and $p_t$ result in a constant
error for low $\varepsilon$, corresponding to $C \mu^{p_x + 1}$ or
$C \eta^{p_t + 1}$, respectively. For larger $p_x$ and $p_t$, these
terms are presumably smaller than $\varepsilon^2$, in the range
considered.
One can furthermore observe that
the choice of $p_x$ does not seem to influence the error otherwise,
while the initial convergence happens at different $\varepsilon$-values
when varying $p_t$.
\begin{figure}[h!]
  \centering
  \includegraphics[width=.85\textwidth]{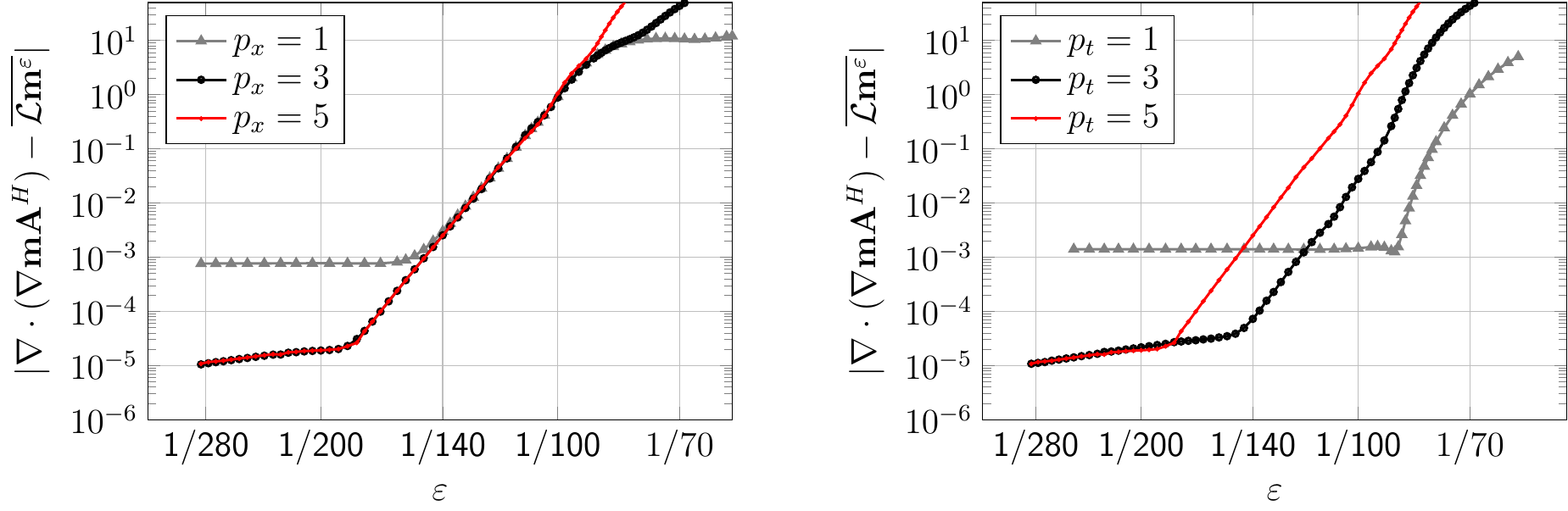}
  \caption{Approximation error for the model (M2) for varying
    $\varepsilon$ with kernel parameters $p_x = 5, q_x = 7, p_t = 5$
    and $q_t = 7$, if not explicitly stated otherwise in the
    plot. Averaging domain sizes $\eta = 1.5\cdot10^{-4}$ and
    $\mu = 0.03$ and damping $\alpha = 0.1$.}
  \label{fig:1d_px_pt}
\end{figure}

Finally, we investigate the influence of the choice of box sizes
$\mu$ and $\eta$ on the error given a fixed value
$\varepsilon= 1/140$. In \cref{fig:1d_box}, it is shown that when
increasing $\mu$ from a small value, there is some initial decrease
in the error due to the reduction in $(\varepsilon/\mu)^{q_x +2}$
before it takes a constant value, due to the fact that the other
terms in \cref{eq:err_est_summary} dominate. At some point,
depending on $p_x$, the error starts increasing again since the term
$\mu^{p_x + 1}$ starts dominating the error. When varying $\eta$, we
have a similar behavior.  However, increasing $\eta$ results in a
much larger initial error reduction, given that $p_t$ is chosen
large enough. The slopes of this decrease depend on $q_t$. Once
$\eta$ is larger than a certain threshold, the error takes a
constant value.  When $p_t = 1$, the error only decreases initially
and then starts increasing again due to the term $\eta^{p_t + 1}$.

\begin{figure}[h!]
  \centering
  \includegraphics[width=.85\textwidth]{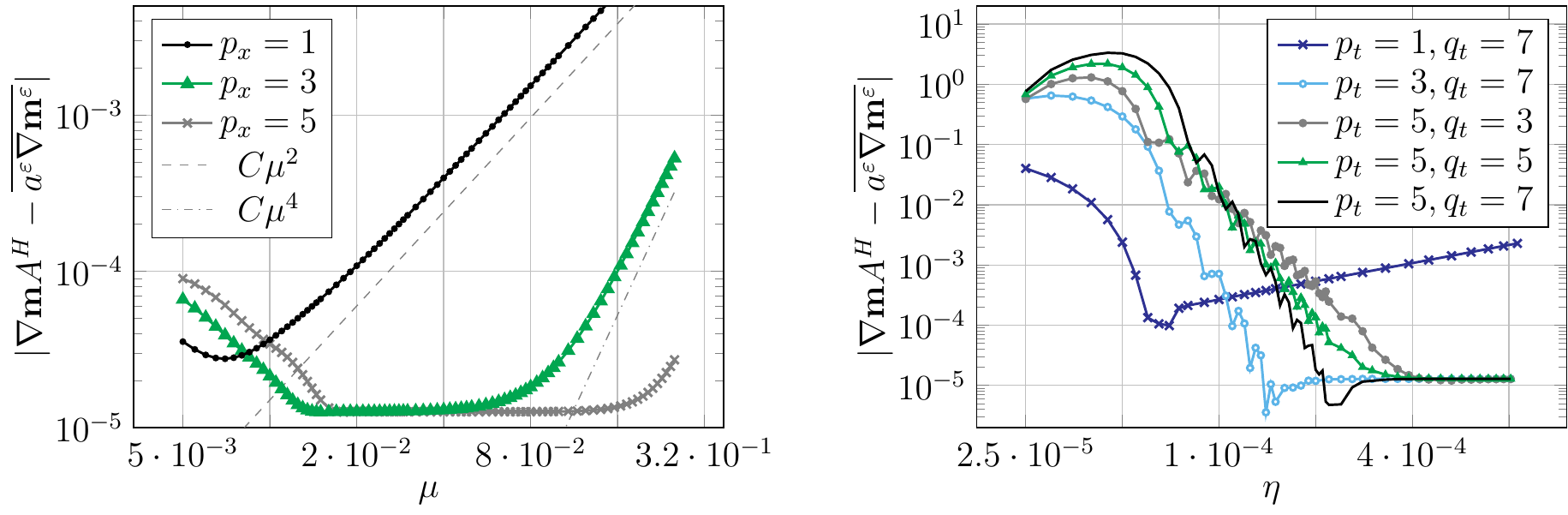}
  \caption{Influence of the box sizes $\eta$ and $\mu$ on the error
    in (M1) for fixed $\varepsilon = 1/140$. Kernel parameters $p_x = 5, q_x = 7, p_t = 5, q_t = 7$
    if not explicitly stated otherwise in the plot. Damping $\alpha = 0.01$.
    When varying $\eta$, $\mu = 0.03$ and when varying $\mu$, $\eta = 1.5\cdot10^{-4}$.
  }
  \label{fig:1d_box}
\end{figure}

Overall, we can conclude that for the 1D example problem, the error
is considerably more affected by temporal than spatial averaging.
The estimate in \Cref{thm:hmm_error} matches conceptually well with
the observed behavior but is slightly too pessimistic.

\subsection{Two-dimensional examples}

We here consider a periodic problem where the material coefficient is chosen to be
\begin{align*}
  a(x_1, x_2) &= \frac{1}{2} +  \left(\frac{1}{2} + \frac{1}{4} \sin(2 \pi x_1)\right)\left(\frac{1}{2} + \frac{1}{4} \sin(2 \pi x_2)\right)
  \\& \hspace{2cm}+ \frac{1}{4} \cos(2 \pi (x_1-x_2)) + \frac{1}{2} \sin(2 \pi x_1)\,.
\end{align*}
Solving the cell problem \cref{eq:cell_problem} numerically, the
corresponding homogenized coefficient is computed to be
\begin{align*}
  \A^H =
  \begin{bmatrix}
    0.61720765 & 0.02618130 \\
    0.02618130 & 0.71523722
  \end{bmatrix}\,,
\end{align*}
a full matrix with two different diagonal elements. The upscaling
errors when varying $\varepsilon$ in (M1), (M2) and (M3) for an
example problem with this material coefficient is shown in
\Cref{fig:2d_comp}.

One can observe a similar behavior as for the 1D problem. However,
note that the error in (M1) is considerably lower than for (M2) and
(M3) in this example. In (M1) and (M2), we again observe convergence
proportional to $\varepsilon^2$ for low values of $\varepsilon$
instead of $\varepsilon$ as suggested by \Cref{thm:hmm_error}.
However, the error in (M3) with low damping, $\alpha = 0.01$, decays
only proportionally to $\varepsilon$ rather than $\varepsilon^2$. We
suspect that this is related to the term $\varepsilon III_4$ in the
analysis of the error in (M3), \cref{eq:nonlinear_est4}, the term
taking the interaction of fast oscillations in time with each other
into account. With higher $\alpha$, the temporal oscillations get
damped away faster and we do not observe that behavior.
\begin{figure}[h]
  \centering
  \includegraphics[width=\textwidth]{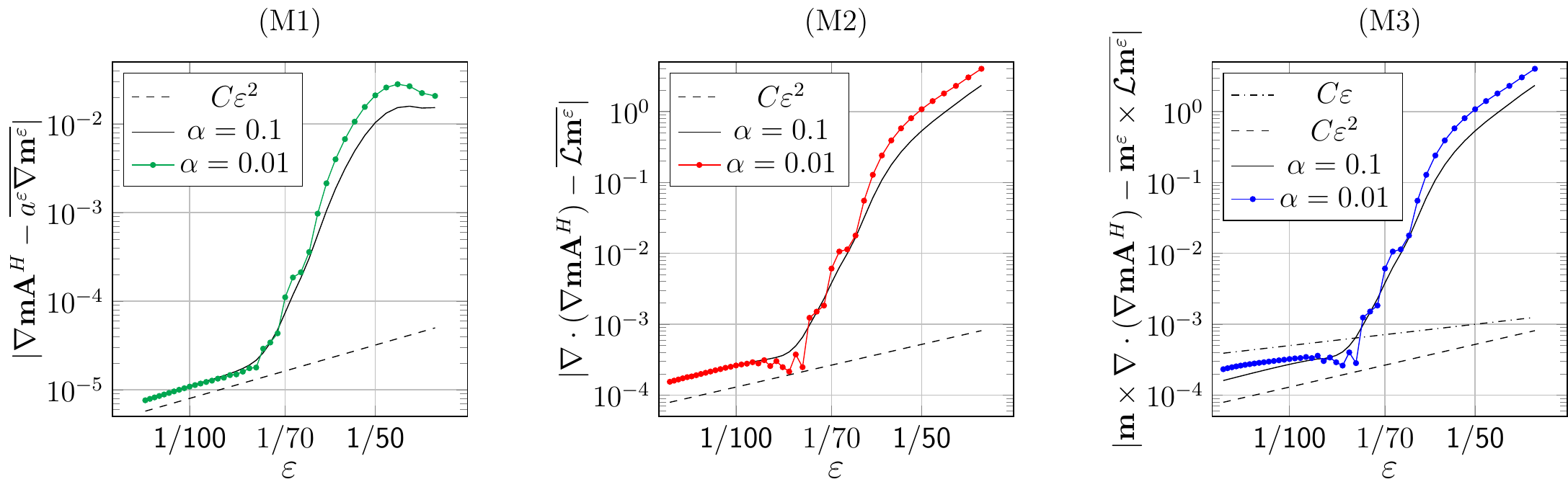}
  \caption{Approximation error in (M1)-(M3) two dimensions for varying $\varepsilon$, where
    $\mu = 0.06, \eta = 0.0003$ and the kernel parameters are
    $ p_x = 5, q_x = 7, p_t = 3$ and $q_t = 7$.}
  \label{fig:2d_comp}
\end{figure}

Apart from this observation for low $\varepsilon$, the errors in
(M2) and (M3) behave very similar when varying the parameters in the
model. We therefore focus on comparing (M1) and (M2) in the
following.
The influence of the kernel parameters is similar to the 1D
problem. As shown in \cref{fig:2d_pt_px_qt}, choosing low $p_x$ or
$p_t$ results in constant error when decreasing $\varepsilon$,
corresponding to $C \mu^{p_x +1}$ or $C\eta^{p_t +1}$, respectively.
The parameter $q_t$ determines the speed of the initial
decay. However, in contrast to the 1D case it is harder to
specifically determine the slopes in this example.
\begin{figure}[h!]
  \centering
  \includegraphics[width=.9\textwidth]{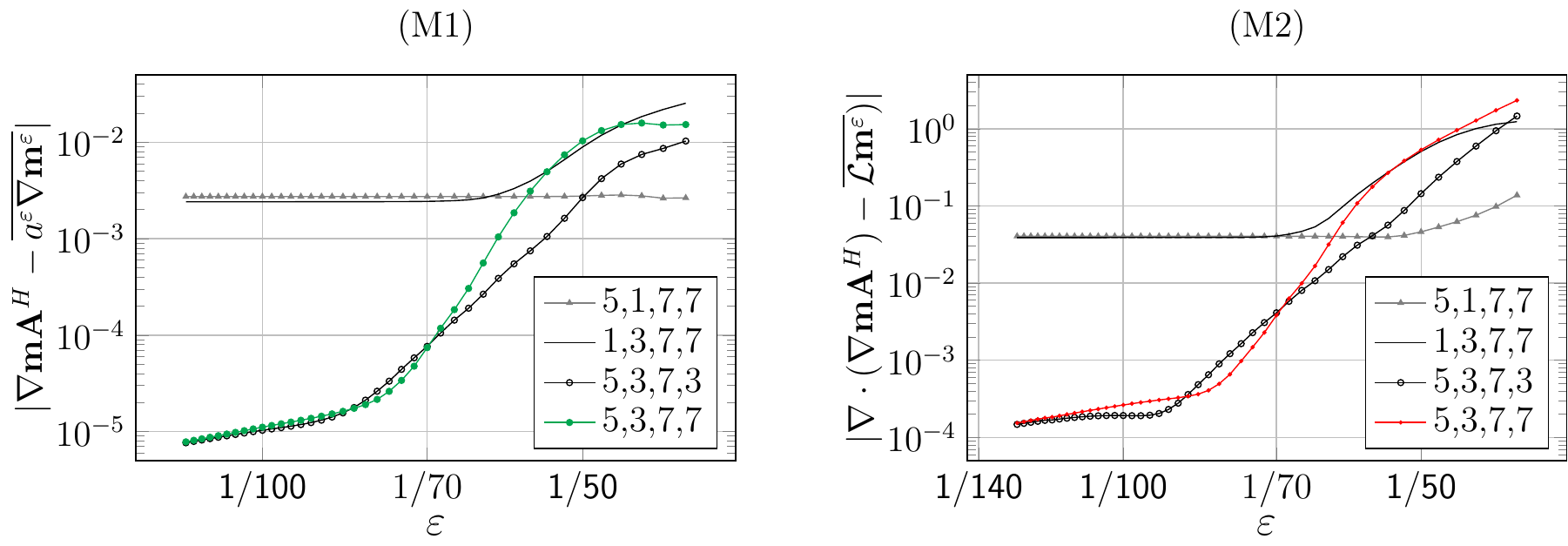}
  \caption{Approximation error in (M1) and (M2) when varying
    $\varepsilon$. Damping $\alpha = 0.1$, spatial averaging
    parameter $\mu = 0.06$ and averaging time $\eta =
    0.0003$. Kernel parameters according to legend, in order
    $p_x, p_t, q_x, q_t$.}
  \label{fig:2d_pt_px_qt}
\end{figure}

In \Cref{fig:eta2d}, the error in (M1) and (M2) when varying $\eta$
is shown for two different values of $\varepsilon$, similar to
\Cref{fig:1d_box}, right, in the 1D case. When choosing low $\eta$,
the errors are high but decrease rapidly as $\eta$ increases. From
$\eta \approx 2 \varepsilon^2$ the error stays at a constant level.

\begin{figure}[h!]
  \centering
  \includegraphics[width=.9\textwidth]{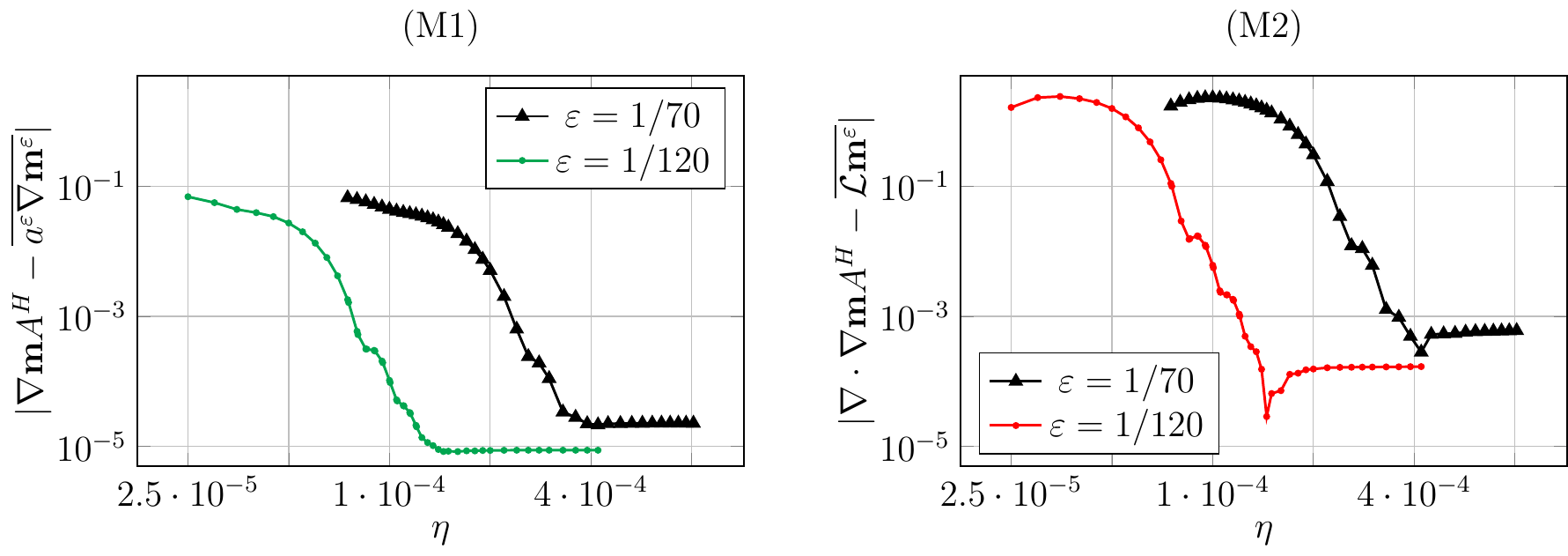}
  \caption{Approximation error in (M1) and (M2) when increasing
    $\eta$. Damping $\alpha = 0.01$, kernel parameters
    $p_x = 5, q_x = 7, p_t = 3, q_t = 7$ and $\mu = 0.06$.}
  \label{fig:eta2d}
\end{figure}

\begin{figure}[h!]
  \centering
  \includegraphics[width=.9\textwidth]{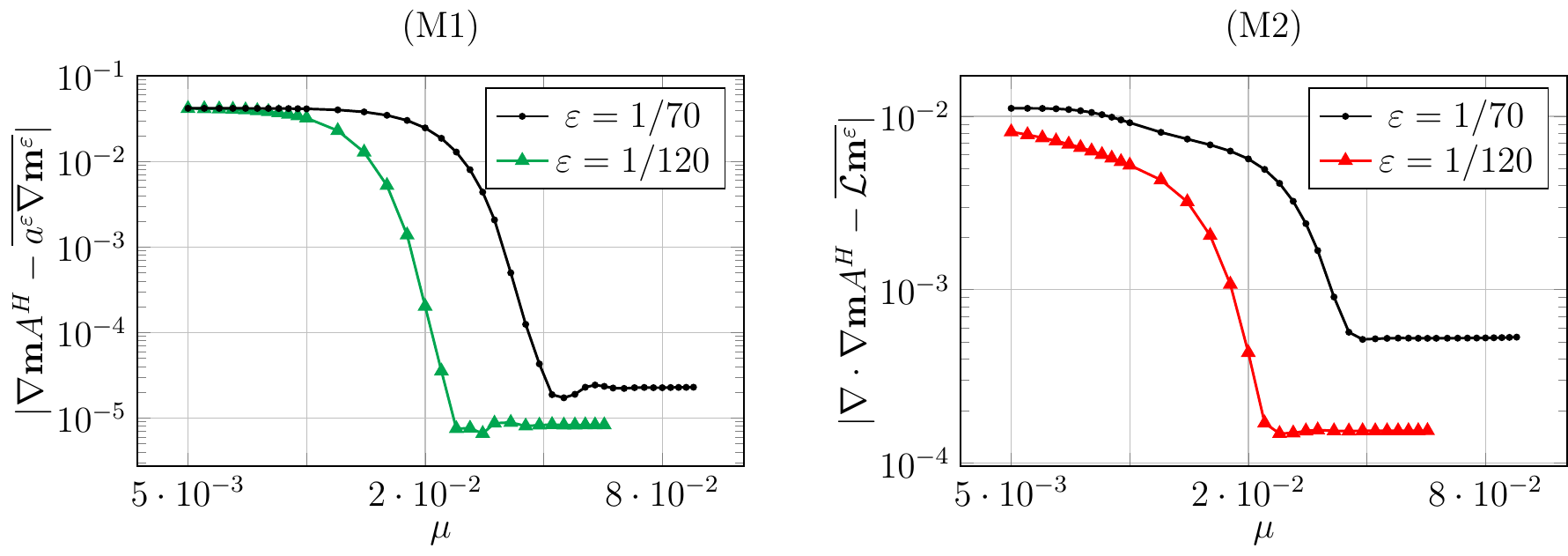}
  \caption{Approximation error in (M1) and (M2) when increasing the
    spatial averaging size $\mu$. Damping $\alpha = 0.01$ and kernel
    parameters $p_x = 5, q_x = 7, p_t = 3, q_t = 7$. For
    $\varepsilon = 1/70$, final time $\eta = 4.5\cdot10^{-4}$ and
    for $\varepsilon = 1/120$, $\eta = 2\cdot10^{-4}$.}
  \label{fig:mu2d}
\end{figure}

Finally, we investigate the influence of $\mu$ as shown in
\cref{fig:mu2d}. We can observe rapidly decreasing errors until
$\mu \approx 3 \varepsilon$, then the errors are almost constant. In
contrast to the 1D case, shown in \Cref{fig:1d_box}, left, the choice of $\mu$ has a significant
impact on the error in this example. In particular, in (M1) the
magnitude of the error is determined by $\eta$ and $\mu$ equally. In
case of (M2), $\eta$ still has a somewhat larger impact than $\mu$.

\section*{Acknowledgments}
Support by the Swedish Research Council under grant no 2017-04579 is
gratefully acknowledged.

\bibliographystyle{siamplain}
\bibliography{hmm}
\end{document}